\documentclass[11pt]{article}
\usepackage[T1]{fontenc}
\usepackage[english]{babel}
\usepackage{setspace}

\newcommand{\Keywords}[1]{\par{\small{\em Keywords\/}: #1}}
\usepackage{abstract}

\usepackage{amssymb}
\usepackage{amsmath}
\usepackage{amsfonts}
\usepackage{amsthm}
\usepackage{lineno}
\usepackage{multicol}
\usepackage{pstricks}
\usepackage{pst-eucl}
\usepackage{graphicx}
\usepackage{subfig}

\usepackage[titletoc]{appendix}

\usepackage[]{natbib}
\bibpunct{(}{)}{;}{}{,}{,}

\usepackage{dsfont}
\newtheorem{thm}{Theorem}[section]
\newtheorem{cor}[thm]{Corollary}
\newtheorem{rem}[thm]{Remark}
\newtheorem{prop}[thm]{Proposition}
\newtheorem{lem}[thm]{Lemma}
\newcommand{\ben}{\vspace{0mm}\begin{equation}}
\newcommand{\een}{\vspace{0mm}\end{equation}}
\newcommand{\be}{\vspace{0mm}\begin{equation*}}
\newcommand{\ee}{\vspace{0mm}\end{equation*}}
\newcommand{\ba}{\vspace{0mm}\begin{equation*}\begin{aligned}}
\newcommand{\ea}{\vspace{0mm}\end{aligned}\end{equation*}}
\newcommand{\ban}{\vspace{0mm}\begin{equation}\begin{aligned}}
\newcommand{\ean}{\vspace{0mm}\end{aligned}\end{equation}}

\numberwithin{equation}{section}

\usepackage{authblk}

\title{\begin{center}\begin{LARGE}
Slow-fast stochastic diffusion dynamics and quasi-stationary distributions for diploid populations
\end{LARGE}
\end{center}}
\date{13/11/2012}
\author[1]{Camille Coron \thanks{coron@cmap.polytechnique.fr, corresponding author}}
\affil[1]{CMAP, \'Ecole Polytechnique, CNRS UMR 7641, Route de Saclay, 91128
Palaiseau Cedex, France}
\date{}
\begin{document}
\fontfamily{ptm} \selectfont
\maketitle

\begin{abstract}
We are interested in the long-time behavior of a diploid population with sexual reproduction, characterized by its genotype composition at one bi-allelic locus. The population is modeled by a $3$-dimensional birth-and-death process with competition, cooperation and Mendelian reproduction. This stochastic process is indexed by a scaling parameter $K$ that goes to infinity, following a large population assumption. When the birth and natural death parameters are of order $K$, the sequence of stochastic processes indexed by $K$ converges toward a slow-fast dynamics. We indeed prove the convergence toward $0$ of a fast variable giving the deviation of the population from Hardy-Weinberg equilibrium, while the sequence of slow variables giving the respective numbers of occurrences of each allele converges toward a $2$-dimensional diffusion process that reaches $(0,0)$ almost surely in finite time. We obtain that the population size and the proportion of a given allele converge toward a generalized Wright-Fisher diffusion with varying population size and diploid selection. Using a non trivial change of variables, we next study the absorption of this diffusion and its long time behavior conditioned on non-extinction. In particular we prove that this diffusion starting from any non-trivial state and conditioned on not hitting $(0,0)$ admits a unique quasi-stationary distribution. We finally give numerical approximations of this quasi-stationary behavior in three biologically relevant cases: neutrality, overdominance, and separate niches.
\end{abstract}



\section{Introduction}\label{sectionintro}

We study the diffusion limit and quasi-stationary behavior of a population of diploid individuals modeled by a non-linear $3$-type birth-and-death process with competition, cooperation and Mendelian reproduction. Individuals are characterized by their genotype at one locus for which there exist $2$ alleles, $A$ and $a$. We study the genetic evolution of the population, i.e. the dynamics of the respective numbers of individuals with genotype $AA$, $Aa$, and $aa$. Following an infinite population size approximation (see also \cite{FournierMeleard2004} and \cite{Champagnat2006} for instance) we assume that the initial number of individuals is of order $K$ where $K$ is a scale parameter that will go to infinity. The population is then modeled by a $3$-type birth-and-death process denoted by $\nu^K=(\nu^K_t, t\geq0)$ and we consider the sequence of stochastic processes $Z^K=\nu^K/K$. At each time $t$ and for all $K$, we define the deviation $Y^K_t$ of the population $Z^K_t$ from a so-called Hardy-Weinberg equilibrium. We are interested in the convergence of the sequence of stochastic processes $Z^K$ when the individual birth and natural death rates are assumed to be both equivalent to $\gamma K$, with $\gamma>0$ (see Section \ref{sectionconvergencediffusion} and \cite{ChampagnatFerriereMeleard2006} for a biological interpretation). In Section \ref{sectionconvergencediffusion} we first establish some conditions on the competition and cooperation parameters so that the sequence of population sizes satisfies a moment propagation property. Next, we prove the convergence of the sequence of stochastic processes $Z^K$ toward a slow-fast dynamics (see \cite{MeleardTran2012} or \cite{Balletal2006} for other examples of such dynamics and \cite{Kurtz1992} and \cite{BerglundGentz2005} for treatments of slow-fast scales in diffusion processes). More precisely, we prove that for all $t>0$, the sequence of random variables $(Y^K_t)_{K\in\mathbb{N}^*}$ goes to $0$ when $K$ goes to infinity, while the sequence of processes $(N^K_t,X^K_t)_{t\geq0}$ giving respectively the population size and the proportion of allele $A$ converges in law toward a "slow" $2$-dimensional diffusion process $(N_t,X_t)_{t\geq0}$. This limiting diffusion $(N,X)$ can be seen as a generalized Wright-Fisher diffusion with varying population size and diploid selection. In Section \ref{sectionQSD}, we first find an appropriate change of variables $S=(f_1(N,X),f_2(N,X))$ such that $S$ is a Kolmogorov diffusion process evolving in a subset $\mathcal{D}$ of $\mathbb{R}^2$. We prove that the stochastic process $S$ is absorbed in the set $\mathbf{A}\cup\mathbf{B}\cup\mathbf{0}$ almost surely in finite time, where $\mathbf{A}$, $\mathbf{B}$ and $\mathbf{0}$ correspond respectively to the sets where $X=1$ (fixation of allele $A$), $X=0$ (fixation of allele $a$), and $N=0$ (extinction of the population). Next, following \cite{CattiauxCollet...2009} and \cite{CattiauxMeleard2009}, we study the quasi-stationary behavior of the diffusion process $(S_t)_{t\geq0}$ conditioned on the non extinction of the population, i.e. on not reaching $\mathbf{0}$. First, the diffusion process $(S_t)_{t\geq0}$ conditioned on not reaching $\mathbf{A}\cup\mathbf{B}\cup\mathbf{0}$ admits a Yaglom limit. Second, if $S_0\notin\mathbf{A}\cup\mathbf{B}\cup\mathbf{0}$ then the law of $S_t$ conditioned on $\{S_t\notin\mathbf{0}\}$ converges when $t$ goes to infinity toward a distribution which is independent from $S_0$. Finally in Section \ref{sectionnumerical}, we present numerical applications and study the long-time coexistence of the two alleles, in three biologically relevant cases: a pure competition neutral case, a case in which each genotype has its own ecological niche, and an overdominance case. In particular, we show that a long-term coexistence of alleles is possible even in some full competition cases, which is not true for haploid clonal reproduction (\cite{CattiauxMeleard2009}). Note that for the sake of simplicity, most proofs of this article are given in the main text for the neutral case where demographic parameters do not depend on the types of individuals, and the calculations for the non-neutral case are given in Appendix \ref{appendixQnonneutre}. 

\section{Model and deterministic limit}\label{sectionmodel}
\subsection{Model}
We consider a population of diploid hermaphroditic individuals characterized by their genotype at one bi-allelic locus, whose alleles are denoted by $A$ and $a$. Individuals can then have one of the three possible genotypes $AA$, $Aa$, and $aa$, (also called types $1$, $2$, and $3$). The population at any time $t$ is represented by a $3$-dimensional birth-and-death process giving the respective numbers of individuals with each genotype. As in \cite{FournierMeleard2004}, \cite{ChampagnatMeleard2011} or \cite{ColletMeleardMetz2012}, we consider an infinite population size approximation. To this end we introduce a scaling parameter $K\in\mathbb{N}^*$ that will go to infinity, and we denote by $\nu^K=((\nu^{1,K}_t,\nu^{2,K}_t,\nu^{3,K}_t),t\geq0)$ the population indexed by $K$. The initial numbers of individuals of each type $\nu_0^{1,K}$, $\nu_0^{2,K}$ and $\nu_0^{3,K}$ will be of order $K$ and we then consider the sequence of rescaled stochastic processes \be\left(Z^K_t\right)_{t\geq0}=\left(Z^{1,K}_t,Z^{2,K}_t,Z^{3,K}_t\right)
_{t\geq0}=\left(\frac{\nu^K_t}{K}\right)_{t\geq0}\ee that gives at each time $t$ the respective numbers of individuals weighted by $1/K$, and with genotypes $AA$, $Aa$, and $aa$. The rescaled population size at time $t$ is denoted by \ben\label{taille} N^K_t=Z^{1,K}_t+Z^{2,K}_t+Z^{3,K}_t\in\frac{\mathbb{Z}_+}{K},\een and the proportion of allele $A$ at time $t$ is denoted by 
\ben\label{proportion} X^K_t=\frac{2Z^{1,K}_t+Z^{2,K}}{2(Z^{1,K}_t+Z^{2,K}_t+Z^{3,K}_t)}.\een As in \cite{Coron2012} or \cite{ColletMeleardMetz2012}, the jump rates of $Z$ model Mendelian panmictic reproduction. More precisely, if we set $e_1=(1,0,0)$, $e_2=(0,1,0)$ and $e_3=(0,0,1)$, then for all $i\in\{1,2,3\}$, the rates $\lambda_i^K(Z)$ at which the stochastic process $Z^K$ jumps from $z=(z_1,z_2,z_3)\in\left(\frac{\mathbb{Z}_+}{K}\right)^3$ to $z+e_i/K$, as long as $z_1+z_2+z_3=n\neq0$, are given by:\ban\label{birthratesQSD} \lambda_1^K(z)&=\frac{Kb_1^K}{n}\left(z_1+\frac{z_2}{2}\right)^2,\\
\lambda_2^K(z)&=\frac{Kb_2^K}{n}2\left(z_1+\frac{z_2}{2}\right)\left(z_3+\frac{z_2}{2}\right),\\
\lambda_3^K(z)&=\frac{Kb_3^K}{n}\left(z_3+\frac{z_2}{2}\right)^2.\ean
These birth rates are naturally set to $0$ if $n=0$ and the demographic parameters $b_i^K\in\mathbb{R}_+$ are called birth demographic parameters. Now individuals can die naturally and either compete or cooperate with other individuals, depending on the genotype of each individual. More precisely, for all $i\in\{1,2,3\}$, the rates $\mu_i^K(z)$ at which the stochastic process $Z^K$ jumps from $z=(z_1,z_2,z_3)\in(\mathbb{Z}_+)^3/K$ to $z-e_i/K$ for $i\in\{1,2,3\}$ are given by:\ban\label{deathratesQSD} \mu_1^K(z)&=Kz_1(d_1^K+K(c_{11}^Kz_1+c_{21}^Kz_2+c_{31}^Kz_3))^+,\\
\mu_2^K(z)&=Kz_2(d_2^K+K(c_{12}^Kz_1+c_{22}^Kz_2+c_{32}^Kz_3))^+,\\
\mu_3^K(z)&=Kz_3(d_3^K+K(c_{13}^Kz_1+c_{23}^Kz_2+c_{33}^Kz_3))^+.\ean
where the interaction (competition or cooperation) demographic parameters $c_{ij}^K$ are arbitrary real numbers and $(x)^+=max(x,0)$ for any $x\in\mathbb{R}$. If $c_{ij}^K>0$ (resp. $c_{ij}^K<0$), then individuals with type $i$ have a negative (resp. positive) influence on individuals of type $j$. The demographic parameter $d_i^K\in\mathbb{R}_+$ is called the intrinsic death rate of individuals of type $i$. From now on, we say that the stochastic process $Z^K$ is neutral for a given $K\in\mathbb{N}^*$ if its demographic parameters do not depend on the types of individuals, i.e. \ben\label{neutre} b_i^K=b^K, \quad d_i^K=d^K\quad \text{ and } \quad c_{ij}^K=c^K\quad\forall i,j\in\{1,2,3\}.\een Note that for any fixed $K\in\mathbb{N}^*$, the pure jump process $Z^K$ is well defined for all $t\in\mathbb{R}_+$. Indeed, $N^K$ is stochastically dominated by a rescaled pure birth process $\overline{N}^K$ that jumps from $n\in\mathbb{Z}_+/K$ to $n+1/K$ at rate $(\underset{i}{\max}\;b_i^K)Kn$ and, from Theorem $10$ in \cite{VillemonaisMeleard2012}, $\overline{N}^K$ does not explode almost surely. The stochastic process $Z^K$ is then a $\frac{(\mathbb{Z}_+)^3}{K}$-valued pure jump Markov process absorbed at $(0,0,0)$, defined for all $t\geq0$ by \be Z^K_t=Z^K_0+\sum_{i\in\{1,2,3\}}\left[\int_0^t\frac{e_i}{K}\mathbf{1}_{\{\theta\leq \lambda_i^K(Z^K_{s^-})\}}\eta_{1}^i(ds,d\theta)-\int_0^t\frac{e_i}{K}\mathbf{1}_{\{\theta\leq \mu_i^K(Z^K_{s^-})\}}\eta_{2}^i(ds,d\theta)\right]\ee where the measures $\eta_{j}^i$ for $i\in\{1,2,3\}$ and $j\in\{1,2\}$ are independent Poisson point measures on $(\mathbb{R}_+)^2$ with intensity $dsd\theta$. For any $K$, the law of $Z^K$ is therefore a probability measure on the trajectory space $\mathbb{D}(\mathbb{R}_+,(\mathbb{Z}_+)^3/K)$ which is the space of left-limited and right-continuous functions from $\mathbb{R}_+$ to $(\mathbb{Z}_+)^3/K$, endowed with the Skorohod topology. Finally, the extended generator $L^K$ of $Z^K$ satisfies for every bounded measurable function $f$ from $(\mathbb{Z}_+)^3/K$ to $\mathbb{R}$ and for every $z\in\frac{(\mathbb{Z}_+)^3}{K}$: 
\ban \label{generateur}L^Kf(z)&
=\underset{i\in\{1,2,3\}}{\sum}\left[\lambda_i^K(z)\left(f\left(z+\frac{e_i}{K}\right)-f(z)\right)+\mu_i^K(z)\left(f\left(z-\frac{e_i}{K}\right)-f(z)\right)\right].\ean 

To end with the model description, let us introduce for all $K\in\mathbb{N}^*$ the stochastic processes $Y^K$ such that for every $t\geq0$, as long as $N^K_t>0$,\ben\label{defY}Y^K_t=\frac{4Z^{1,K}_tZ^{3,K}_t-(Z^{2,K}_t)^2}{4N^K_t}.\een If $N^K_t=0$, we set $Y^K_t=0$ as $\vert Y^K_t\vert\leq N^K_t$ for all $t\geq0$. This stochastic process will play a main role in the article and note first that: $$Y^K_t=Z^{1,K}_t-\frac{(2Z^{1,K}_t+Z^{2,K}_t)^2}{4N^K_t}=\left(p^{AA,K}_t-(p^{A,K}_t)^2\right)N^K_t$$ if $p^{A,K}_t$ (resp. $p^{AA,K}_t$) is the proportion of allele $A$ (resp. genotype $AA$) in the population at time $t$. Similarly, $$Y^K_t=\left(p^{aa,K}_t-(p^{a,K}_t)^2\right)N^K_t=\left(2p^{A,K}_tp^{a,K}_t-p^{Aa,K}_t\right)N^K_t.$$ Then if $Y^K_t=0$, the proportion of each genotype in the population $Z^K_t$ is equal to the proportion of pairs of alleles forming this genotype. By an abuse of language, if $Y^K_t=0$ we say that the population $Z^K_t$ is at Hardy-Weinberg equilibrium (\cite{CrowKimura1970}, p. $34$). In the rest of the article, we will see that the quantities of interest in this model are the population size $N^K$, the deviation from Hardy-Weinberg equilibrium $Y^K$ and the proportion $X^K$ of allele $A$. More precisely, the following lemma gives the change of variable:

\begin{lem}\label{cgtvar}
Let us set for all $z=(z_1,z_2,z_3)\in(\mathbb{R}_+)^3\setminus\{(0,0,0)\}$, \ba \phi_1(z)&=z_1+z_2+z_3,\quad\phi_2(z)=\frac{2z_1+z_2}{2(z_1+z_2+z_3)}\quad\text{and}\quad\phi_3(z)=\frac{4z_1z_3-(z_2)^2}{4(z_1+z_2+z_3)}.\ea
Then the function \ba \phi:(\mathbb{R}_+)^3\setminus\{(0,0,0)\}&\rightarrow E\\z&\mapsto \phi(z)=(\phi_1(z),\phi_2(z),\phi_3(z))\ea where $E=\{(n,x,y)|n\in\mathbb{R}_+^*,x\in[0,1],-n\min(x^2,(1-x)^2)\leq y\leq nx(1-x)\}$ is a bijection.
\end{lem}

Note that $\phi(Z^K)=(N^K,X^K,Y^K),$ where $N^K$, $X^K$, and $Y^K$ have been respectively defined in Equations \eqref{taille}, \eqref{proportion} and \eqref{defY}. 

\begin{proof} We easily obtain that $(n,x,y)=\phi(z_1,z_2,z_3)$ if and only if 
\ben\label{inverse} z_1=nx^2+y,\quad z_2=2nx(1-x)-2y\quad\text{and}\quad z_3=n(1-x)^2+y,\een which gives the injectivity. Next, for any $(n,x,y)$ such that $n\in\mathbb{R}_+^*$, $x\in[0,1]$ and $-n\min(x^2,(1-x)^2)\leq y\leq nx(1-x)$, $z=(z_1,z_2,z_3)$ defined by Equation \eqref{inverse} is in $(\mathbb{R}_+)^3\setminus\{(0,0,0)\}$ which gives the surjectivity.
\end{proof}

\subsection{Convergence toward a deterministic system}\label{ssectiondeterm}

This section aims at understanding the behavior of the population when the birth and natural death parameters do not depend on $K$. The results obtained at this scaling will indeed give an intuition of the behavior of the population when $b^{i,K}$ and $d^{i,K}$ are of order $K$, which is studied in Section \ref{sectionconvergencediffusion}. In particular, we prove in this section a long-time convergence of the population toward Hardy-Weinberg equilibrium. 

We consider a particular case of the scaling considered in Section $3$ of \cite{ColletMeleardMetz2012}. More precisely, we set: \ba
b_i^K&=\beta\in\mathbb{R}_+^*\\
d_i^K&=\delta\in\mathbb{R}_+\\
Kc_{ij}^K&=\alpha\in\mathbb{R}_+\\
Z^K_0&\underset{K \to \infty}{\longrightarrow} \mathcal{Z}_0\quad\text{in law,}
\ea where $\mathcal{Z}_0$ is a deterministic vector of $(\mathbb{R}_+)^3$. Note that the process $Z^K$ is neutral for all $K\in\mathbb{N}^*$. For all $i\in\{1,2,3\}$ and $z=(z_1,z_2,z_3)\in(\mathbb{R}_+)^3$ we now denote by $\lambda^{\infty}_i(z)$ (resp. $\mu^{\infty}_i(z)$) the limit of the rescaled birth (resp. death) rate (see Equations \eqref{birthratesQSD} and \eqref{deathratesQSD}): $$\lambda_i^{\infty}(z)=\underset{K\rightarrow\infty}{\lim}\frac{\lambda_i^K(z)}{K}\quad\text{and}\quad\mu_i^{\infty}(z)=\underset{K\rightarrow\infty}{\lim}\frac{\mu_i^K(z)}{K}.$$ For instance, if we set $(n,x,y)=\phi(z)$ where $\phi$ has been defined in Lemma \ref{cgtvar}, we get $$\lambda_1^{\infty}(z)=\beta nx^2\quad\text{and}\quad \mu_1^{\infty}(z)=(\delta+\alpha n)z_1.$$ Here, Proposition $3.2$ in \cite{ColletMeleardMetz2012} (see also Theorem $5.3$ of \cite{FournierMeleard2004}) gives that for all $T>0$, the sequence of stochastic processes $(Z^K_t,t\in[0,T])_{K\in\mathbb{Z}_+^*}$ converges in law in $\mathbb{D}([0,T],(\mathbb{R}_+)^3)$ toward a deterministic limit $\mathcal{Z}=(\mathcal{Z}^1,\mathcal{Z}^2,\mathcal{Z}^3)$, which is the unique continuous solution of the differential system: \ben \label{systemedeterministe}
\left\{\begin{tabular}{l}
$\frac{d\mathcal{Z}^1_t}{dt}=\lambda_1^{\infty}(\mathcal{Z}_t)-\mu_1^{\infty}(\mathcal{Z}_t)$\\ 
$\frac{d\mathcal{Z}^2_t}{dt}=\lambda_2^{\infty}(\mathcal{Z}_t)-\mu_2^{\infty}(\mathcal{Z}_t)$\\ 
$\frac{d\mathcal{Z}^3_t}{dt}=\lambda_3^{\infty}(\mathcal{Z}_t)-\mu_3^{\infty}(\mathcal{Z}_t).$ 
\end{tabular}\right. \een A solution of this system does not appear immediately, but using the change of variables $\phi$ introduced in Lemma \ref{cgtvar}, we obtain that $\phi(\mathcal{Z})=(\mathcal{N},\mathcal{X},\mathcal{Y})$ satisfies the

\begin{prop}\label{propsolutiondeterministe}
\begin{description}
\item[$(i)$] If $\beta=\delta$ then \ben\label{formuleNdeterm1} \mathcal{N}_t=\frac{\mathcal{N}_0}{\alpha\mathcal{N}_0t+1},\quad\forall t\geq0.\een Else \ben\label{formuleNdeterm} \mathcal{N}_t=\frac{(\beta-\delta)\mathcal{N}_0e^{(\beta-\delta)t}}{(\beta-\delta)+\alpha\mathcal{N}_0(e^{(\beta-\delta)t}-1)}\quad\forall t\geq0.\een
\item[$(ii)$] For all $t\geq0$, $\mathcal{X}_t=\mathcal{X}_0$.
\item[$(iii)$] If $\alpha=0$ then $\mathcal{Y}_t=\mathcal{Y}_0e^{-\delta t}$ for all $t\geq0.$ 

If $\alpha\neq0$ and $\beta=\delta$ then $ \mathcal{Y}_t=\left(\mathcal{Y}_0-\ln\left(1+\alpha\mathcal{N}_0t\right)\right)e^{-\delta t}$ for all $t\geq0$. 

If $\alpha\neq0$, $\beta\neq\delta$ and $\mathcal{N}_0=\frac{\beta-\delta}{\alpha}$, then $\mathcal{Y}_t=\mathcal{Y}_0e^{-\beta t}$ for all $t\geq0.$ 

Finally if $\alpha\neq0$, $\beta\neq\delta$ and $\mathcal{N}_0\neq\frac{\beta-\delta}{\alpha}$ then for all $t\geq0$ and if $C=\frac{\mathcal{Y}_0}{1-\frac{\alpha\mathcal{N}_0}{\beta-\delta}}$, $\mathcal{Y}_t=Ce^{-\delta t}\left(1-\frac{\alpha\mathcal{N}_0e^{(\beta-\delta)t}}{(\beta-\delta)+\alpha\mathcal{N}_0(e^{(\beta-\delta)t}-1)}\right)$ for all $t\geq0$
\end{description}
\end{prop}
\begin{proof} $\mathcal{N}$ is solution of the logistic equation $d\mathcal{N}_t/dt=(\beta-\delta-\alpha \mathcal{N}_t)\mathcal{N}_t$ whose unique solution is given for $\beta\neq\delta$ and $\alpha\neq0$ in \cite{Verhulst1844}. We then have Equation \eqref{formuleNdeterm} that remains true if $\alpha=0$ and $\beta\neq\delta$. If $\beta=\delta$ we easily find that the unique solution of the equation $d\mathcal{N}_t/dt=-\alpha \mathcal{N}_t^2$ is given by Equation \eqref{formuleNdeterm1}. Therefore, $\mathcal{N}_t>0$ for all $t\geq0$. Then, using the System of equations \eqref{systemedeterministe}, we find that $d\mathcal{X}_t/dt=0$ for all $t\geq0$ which gives $(ii)$. Finally for $(iii)$, $\mathcal{Y}$ is solution of the differential equation $d\mathcal{Y}_t/dt=-(\delta+\alpha\mathcal{N}_t)\mathcal{Y}_t$. If $\alpha=0$ the solution is clear. If $\alpha\neq0$ and $\beta=\delta$ we find the result by looking for a solution of the form $C(t)e^{-\delta t}$ and from Equation \eqref{formuleNdeterm1}. If $\alpha\neq0$, $\beta\neq\delta$ and $\mathcal{N}_0=\frac{\beta-\delta}{\alpha}$, then from Equation \eqref{formuleNdeterm}, $\mathcal{N}_t=\mathcal{N}_0$ for all $t$ and the solution follows. Finally if $\alpha\neq0$, $\beta\neq\delta$ and $\mathcal{N}_0\neq\frac{\beta-\delta}{\alpha}$, looking for a solution of the form $\mathcal{Y}_t=Ce^{-\delta t}+\frac{B(t)e^{(\beta-\delta)t}}{(\beta-\delta)+\alpha\mathcal{N}_0(e^{(\beta-\delta)t}-1)}$ with $B(t)=De^{-\delta t}$ we find the result. 
\end{proof}

\bigskip Note that, in this scaling, the population does no get extinct in finite time and the proportion of allele $A$ remains constant. Besides, $\mathcal{Y}_t$ goes to $0$ when $t$ goes to infinity, which gives a long-time convergence of the population toward Hardy-Weinberg equilibrium. We therefore observe biodiversity conservation but can not study the Darwinian evolution of the population, since none of the two alleles will eventually disappear. These points are due to the fact that the population is neutral and to the large population size assumption (\cite{CrowKimura1970}, p. $34$). 

\section{Convergence toward a slow-fast stochastic dynamics}\label{sectionconvergencediffusion}

In this section, we investigate a new scaling under which the population size and proportion of allele $a$ evolve stochastically with time (in particular the population can get extinct and one of the two alleles can eventually get fixed), while the population still converges rapidly toward Hardy-Weinberg equilibrium. The results we obtain then provide a rigorous justification of the assumption of Hardy-Weinberg equilibrium which is often made when studying large populations. However we will explain that this result does not mean that the genetic composition of a diploid population can always be reduced to a set of alleles. 

We assume that birth and natural death parameters are of order $K$, while $Z^K_0$ converges in law toward a random vector $Z_0$. More precisely, we set for $\gamma>0$: \ba
b^{i,K}&= \gamma K+\beta_i\in[0,\infty[\\
d^{i,K}&= \gamma K+\delta_i\in[0,\infty[\\
c_{ij}^K&=\frac{\alpha_{ij}}{K}\in\mathbb{R}\\
Z^K_0&\underset{K\rightarrow\infty}{\rightarrow} Z_0 \quad\text{ in law,}
\ea where $Z_0$ is a $(\mathbb{R}_+)^3$-valued random variable. This means that the birth and natural death events are now happening faster, which will introduce some stochasticity in the limiting process. The results presented in Proposition \ref{propsolutiondeterministe} suggest that under these conditions, $Y^K$ will be a "fast" variable that converges directly toward the long time equilibrium of $\mathcal{Y}$ (equal to $0$), while $X^K$ and $N^K$ will be "slow" variables, converging toward a non deterministic process. First, we need a moment propagation property. It is not true for all values of the interaction parameters $\alpha_{ij}$ and in particular when $\alpha_{ii}\leq0$ for any $i\in\{1,2,3\}$, i.e. when individuals with a same given genotype cooperate or do not compete. For any $z=(z_1,z_2,z_3)\in(\mathbb{R}_+)^3$, let $g(z)=\underset{i,j\in\{1,2,3\}}{\sum}\alpha_{ij}z_iz_j.$ We establish the following:

\begin{prop}\label{propgpositif}
If $g(z)>0$ for all $z\in\left(\mathbb{R}_+\right)^3\setminus\{(0,0,0)\}$ and if, for any $k\in\mathbb{Z}_+$, there exists a constant $C_0$ such that for all $K\in\mathbb{N}^*$, $\mathbb{E}((N^K_0)^k))\leq C_0$, then
\begin{description}
\item[(i)] There exists a constant $C$ such that $\underset{K}{\sup}\;\underset{t\geq0}{\sup}\;\mathbb{E}((N^K_t)^k))\leq C.$
\item[(ii)] For all $T<+\infty$, there exists a constant $C_T$ such that $\underset{K}{\sup}\;\mathbb{E}\left(\underset{t\leq T}{\sup}\;(N^K_t)^k\right)\leq C_T.$
\end{description}
\end{prop}

\begin{proof} Assume that $g(z)>0$ for all $z\in\left(\mathbb{R}_+\right)^3\setminus\{(0,0,0)\}$. Note that $g$ writes $$g(z)=\phi_1(z)^2\underset{i,j\in\{1,2,3\}}{\sum}\alpha_{ij}p_ip_j:=\phi_1(z)^2f(p_1,p_2,p_3)$$ where $\phi_1$ has been defined in Lemma \ref{cgtvar} and $p_i=z_i/\phi_1(z)$ for all $i\in\{1,2,3\}$. If $g(z)>0$ for all $z\in\left(\mathbb{R}_+\right)^3\setminus\{(0,0,0)\}$, the function $f$ is then non-negative and continuous on $\{(p_1,p_2,p_3)\in[0,1]^3|p_1+p_2+p_3=1\}$ which is a compact set. Then $f$ reaches its minimum $m\geq0$. Now if there exists $(p_1,p_2,p_3)$ such that $f(p_1,p_2,p_3)=0$ then for any $n>0$, $g(np_1,np_2,np_3)=0$ and $(np_1,np_2,np_3)\in\left(\mathbb{R}_+\right)^3\setminus\{(0,0,0)\}$ which is impossible. Then $m>0$, $g(z)\geq m\phi_1(z)^2$, and $\mu_1^K(z)+\mu_2^K(z)+\mu_3^K(z)\geq(\gamma K+\underset{i}{\inf}\delta_i+m\phi_1(z))K\phi_1(z)$ for all $z\in(\mathbb{R}_+)^3$. Then, for all $K$, $N^K$ is stochastically dominated by the logistic birth-and-death process $\overline{N}^K$ jumping from $n\in\mathbb{Z}_+/K$ to $n+1/K$ at rate $(\gamma K+\underset{i\in\{1,2,3\}}{\sup}\;\beta_i)Kn$ and from $n$ to $n-1/K$ at rate $(\gamma K+\underset{i\in\{1,2,3\}}{\inf}\delta_i+m n)Kn$. Finally, the sequence of stochastic processes $\overline{N}^K$ satisfies \textbf{(i)} and \textbf{(ii)}, which gives the result (see respectively Lemma $1$ of \cite{Champagnat2006} and the proof of Theorem $5.3$ of \cite{FournierMeleard2004}). 
\end{proof}

\bigskip From now we assume the following hypotheses: \ben\label{hyp1}\tag{H1} g(z)>0 \quad\text{for all}\quad z\in\left(\mathbb{R}_+\right)^3\setminus\{(0,0,0)\},\een and a $3$-rd-order moments conditions: \ben\label{hyp2}\tag{H2}\text{there exists } C<\infty \quad\text{such that}\quad\underset{K}{\sup}\;\mathbb{E}((N^K_0)^3))\leq C.\een 
In Section \ref{sectionQSD} we will consider only the symmetrical case where $\alpha_{ij}=\alpha_{ji}$ for all $i,j$ and give some explicit sufficient conditions on the parameters $\alpha_{ij}$ so that \eqref{hyp1} is true. 

\bigskip
The following proposition gives that $(Y^K_t,t\geq0)$ is a fast variable that converges toward the deterministic value $0$ when $K$ goes to infinity.

\begin{prop} \label{propY} Under \eqref{hyp1} and \eqref{hyp2}, for all $s,t>0$, $\underset{t\leq u\leq t+s}{\sup}\mathbb{E}((Y^K_u)^2)\rightarrow 0$ when $K$ goes to infinity.
\end{prop}

\begin{proof} Let us fix $z=(z_1,z_2,z_3)\in(\mathbb{R}_+)^3$ and set $(n,x,y)=\phi(z)$ where $\phi$ is defined in Lemma \ref{cgtvar}. The extended generator $L^K$ of the jump process $Z^K$ applied to a measurable real-valued function $f$ (see Equation \eqref{generateur}) is decomposed as follows in $z$: \ban\label{generateurdecompose} L^Kf(z)&=\gamma K^2y\left[f\left(z-\frac{e_1}{K}\right)-2f\left(z-\frac{e_2}{K}\right)+f\left(z-\frac{e_3}{K}\right)\right]\\&+\gamma K^2n\,(x)^2\left[f\left(z+\frac{e_1}{K}\right)+f\left(z-\frac{e_1}{K}\right)-2f\left(z\right)\right]\\&+\gamma K^22nx(1-x)\left[f\left(z+\frac{e_2}{K}\right)+f\left(z-\frac{e_2}{K}\right)-2f\left(z\right)\right]\\&+\gamma K^2n\left(1-x\right)^2\left[f\left(z+\frac{e_3}{K}\right)+f\left(z-\frac{e_3}{K}\right)-2f(z)\right]\\&+\beta_1Kn\,(x)^2\left[f\left(z+\frac{e_1}{K}\right)-f(z)\right]+\beta_2K2nx(1-x)\left[f\left(z+\frac{e_2}{K}\right)-f(z)\right]\\&+\beta_3Kn(1-x)^2\left[f\left(z+\frac{e_3}{K}\right)-f(z)\right]\\&
+K\underset{i\in\{1,2,3\}}{\sum}\left(\delta_i+\underset{j\in\{1,2,3\}}{\sum}\alpha_{ji}z_j\right)z_i\left[f\left(z-\frac{e_i}{K}\right)-f(z)\right]\\&+K\underset{i\in\{1,2,3\}}{\sum}\left(\gamma K+\delta_i+\underset{j\in\{1,2,3\}}{\sum}\alpha_{ji}z_j\right)^-z_i\left[f\left(z-\frac{e_i}{K}\right)-f(z)\right]\ean where $(x)^-=\max(-x,0)$. Now if $f=\left(\phi_3\right)^2$, then there exist functions $g_1^K$, $g_2^K$, and $g_3^K$ and a constant $C_1$ such that for all $z\in(\mathbb{R}_+)^3$, \ba f\left(z+\frac{e_1}{K}\right)-f(z)&=-\frac{2f(z)}{Kn}+\frac{2z_3y}{Kn}+g_1^K(z)\\f\left(z+\frac{e_2}{K}\right)-f(z)&=-\frac{2f(z)}{Kn}-\frac{z_2y}{Kn}+g_2^K(z)\\f\left(z+\frac{e_3}{K}\right)-f(z)&=-\frac{2f(z)}{Kn}+\frac{2z_1y}{Kn}+g_3^K(z)\ea with $\vert g_i^K(z)\vert\leq \frac{C_1}{K^2}$ for all $i\in\{1,2,3\}$. Finally, note that since $\gamma>0$, there exists a positive constant $C_2$ such that for all $i\in\{1,2,3\}$ and all $z\in(\mathbb{R}_+)^3$, \ba\mathbf{1}_{\left\{\left(\gamma K+\delta_i+\underset{j\in\{1,2,3\}}{\sum}\alpha_{ji}z_j\right)^-\neq 0\right\}}&=\mathbf{1}_{\left\{\underset{j\in\{1,2,3\}}{\sum}\alpha_{ji}z_j\leq -\gamma K-\delta_i\right\}}\\&\leq\mathbf{1}_{\left\{\exists j\in\{1,2,3\}\,:\,\alpha_{ji}<0\,,\, \alpha_{ji}z_j\leq -\frac{\gamma}{3}K-\frac{\delta_i}{3}\right\}}\leq\mathbf{1}_{\left\{\phi_1(z)\geq C_2K\right\}}.\ea Therefore, there exists a positive constant $C_3$ such that \ba L^K\left(\phi_3\right)^2(z)
\leq-2\gamma K\left(\phi_3\right)^2(z)+C_3\left[(\phi_1(z))^2(\phi_1(z)+K\mathbf{1}_{\{\phi_1(z)\geq C_2K\}})+1\right].\ea Now from Proposition \ref{propgpositif} and Markov inequality, under \eqref{hyp1} and \eqref{hyp2}, there exists a constant $C$ such that $\underset{K}{\sup}\;\underset{t\geq0}{\sup}\;\mathbb{E}\left(C_3\left[(N^K_t)^2(N^K_t+K\mathbf{1}_{\{N^K_t\geq C_2K\}})+1\right]\right)\leq C.$ Therefore from the Kolmogorov forward equation, since $0\leq (Y^K_t)^2\leq(N^K_t)^2$ for all $t$ and from Proposition \ref{propgpositif}, \ba \frac{d\mathbb{E}\left((Y^K_t)^2\right)}{dt}&\leq -2\gamma K\mathbb{E}\left((Y^K_t)^2\right)+C.\ea
This gives for all $t\geq0$, $\frac{d}{dt}\left(e^{2\gamma Kt}\mathbb{E}\left((Y^K_t)^2\right)\right)\leq Ce^{2\gamma Kt}.$ Then by integration, 
\ba\mathbb{E}\left((Y^K_t)^2\right)&\leq \mathbb{E}\left(\left(Y^K_0\right)^2\right)e^{-2\gamma Kt}+\frac{C}{2\gamma K}-\frac{C}{2\gamma K}e^{-2\gamma Kt}\\&\leq \mathbb{E}\left((N^K_0)^2\right)e^{-2\gamma Kt}+\frac{C}{2\gamma K}-\frac{C}{2\gamma K}e^{-2\gamma Kt},\quad\text{which gives the result.}\ea \end{proof}

\bigskip In particular, under \eqref{hyp1} and \eqref{hyp2} and for all $t>0$, $Y^K_t$ converges in $L^2$ to $0$. We say that $Y^K$ is a fast variable compared to the vector $(N^K,X^K)$ whose behavior is now studied. Let us introduce the following notation for all $z=(z_1,z_2,z_3)\in(\mathbb{R}_+)^3$: $$\psi_1(z)=2z_1+z_2 \quad\text{and}\quad \psi_2(z)=2z_3+z_2.$$ Note that $\psi_1(Z^K_t)=2N^K_tX^K_t$ (resp. $\psi_2(Z^K_t)=2N^K_t(1-X^K_t)$) is the rescaled number of allele $A$ (resp. $a$) in the rescaled population $Z^K$ at time $t$. For any $K\in\mathbb{N}^*$, $(\psi_1(Z^K),\psi_2(Z^K))$ is a pure jump Markov process with trajectories in $\mathbb{D}(\mathbb{R}_+,(\mathbb{Z}_+)^2/K)$ and for all $i\in\{1,2\}$, the process $\psi_i(Z^K)$ admits the following semi-martingale decomposition: for all $t\geq0$,
$$\psi_i(Z^K_t)=\psi_i(Z^K_0)+M^{i,K}_t+\int_0^tL^K\psi_i(Z^K_s)ds$$ where $M^K=(M^{1,K},M^{2,K})$ is, under \eqref{hyp1} and \eqref{hyp2}, a square integrable $\mathbb{R}^2$-valued càd-làg martingale (from Proposition \ref{propgpositif}) and is such that for all $i,j\in\{1,2\}$, the predictable quadratic variation is given for all $t\geq0$ by:
$$\langle M^{i,K},M^{j,K}\rangle_t=\int_0^tL^K\psi_i\psi_j(Z^K_s)
-\psi_i(Z^K_s)L^K\psi_j(Z^K_s)-\psi_j(Z^K_s)L^K\psi_i(Z^K_s)ds.$$
Using this decomposition we prove the

\begin{thm}\label{TheoremConvergenceNX} Under \eqref{hyp1} and \eqref{hyp2}, if the sequence $\{(\psi_1(Z^K_0),\psi_2(Z^K_0))\}_{K\in\mathbb{N^*}}$ of random variables converges in law toward a random variable $(N^A_0,N^a_0)$ when $K$ goes to infinity, then for all $T>0$, the sequence of stochastic processes $(\psi_1(Z^K),\psi_2(Z^K))$ converges in law in $\mathbb{D}([0,T],(\mathbb{R}_+)^2)$ when $K$ goes to infinity, toward the diffusion process $(N^A,N^a)$ starting from $(N^A_0,N^a_0)$ and satisfying the following diffusion equation, where $B=(B^1,B^2)$ is a $2$-dimensional Brownian motion: \ban\label{diffusionA1A2}
dN^A_t&=\frac{N^A_t}{N^A_t+N^a_t}\left[\left[\beta_1-\delta_1-\frac{\alpha_{11}(N^A_t)^2+\alpha_{21}2N^A_tN^a_t+\alpha_{31}(N^a_t)^2}{2(N^A_t+N^a_t)}\right]N^A_t\right.\\&\phantom{\frac{N^A_t}{N^A_t+N^a_t}[[}+\left.\left[\beta_2-\delta_2-\frac{\alpha_{12}(N^A_t)^2+\alpha_{22}2N^A_tN^a_t+\alpha_{32}(N^a_t)^2}{2(N^A_t+N^a_t)}\right]N^a_t\right]dt\\&+\sqrt{\frac{4\gamma}{N^A_t+N^a_t}}N^A_tdB^1_t+\sqrt{2\gamma\frac{N^A_tN^a_t}{N^A_t+N^a_t}}dB^2_t\\
dN^a_t&=\frac{N^a_t}{N^A_t+N^a_t}\left[\left[\beta_3-\delta_3-\frac{\alpha_{33}(N^a_t)^2+\alpha_{23}2N^A_tN^a_t+\alpha_{13}(N^A_t)^2}{2(N^A_t+N^a_t)}\right]N^a_t\right.\\&\phantom{\frac{N^a_t}{N^A_t+N^a_t}[[}+\left.\left[\beta_2-\delta_2-\frac{\alpha_{32}(N^a_t)^2+\alpha_{22}2N^A_tN^a_t+\alpha_{12}(N^A_t)^2}{2(N^A_t+N^a_t)}\right]N^A_t\right]dt\\&+\sqrt{\frac{4\gamma}{N^A_t+N^a_t}}N^a_tdB^1_t-\sqrt{2\gamma\frac{N^A_tN^a_t}{N^A_t+N^a_t}}dB^2_t
\ean\end{thm}

Note that the diffusion coefficients of the diffusion process $((N^A_t,N^a_t),t\geq0)$ do not explode when $N^A_t+N^a_t$ goes to $0$ since $\frac{N^A_t}{\sqrt{N^A_t+N^a_t}}\leq\sqrt{N^A_t+N^a_t}$, $\frac{N^a_t}{\sqrt{N^A_t+N^a_t}}\leq\sqrt{N^A_t+N^a_t}$ and $\frac{N^A_tN^a_t}{N^A_t+N^a_t}\leq N^A_t+N^a_t$.
From this theorem we deduce the convergence of the sequence of processes $$(N^K,X^K)=\left(\frac{\psi_1(Z^K)+\psi_2(Z^K)}{2},\frac{\psi_1(Z^K)}{\psi_1(Z^K)+\psi_2(Z^K)}\right)$$ stopped when $N^K\leq\epsilon$ for any $\epsilon>0$:

\begin{cor}\label{corNX}
For any $\epsilon>0$ and $T>0$, let us define $T_{\epsilon}^K=\inf\{t\in[0,T]:N^K_t\leq\epsilon\}$. If the sequence of random variables $(N^K_0,X^K_0)\in[\epsilon,+\infty[\times[0,1]$ converges in law toward a random variable $(N_0,X_0)\in]\epsilon,+\infty[\times[0,1]$ when $K$ goes to infinity, then the sequence of stopped stochastic processes $\{(N^K,X^K)_{.\wedge T^K_{\epsilon}}\}_{K\geq1}$ converges in law in $\mathbb{D}([0,T],[\epsilon,\infty[\times[0,1])$ when $K$ goes to infinity, toward the stopped diffusion process $(N,X)_{.\wedge T_{\epsilon}}$ ($T_{\epsilon}=\inf\{t\in[0,T]:N_t=\epsilon\}$), starting from $(N_0,X_0)$ and satisfying the following diffusion equation: 
\ban\label{diffusionNXnoneutre}
dN_t\!&=\sqrt{2\gamma N_t}dB^1_t\\&+N_t\left[X_t^2\left(\beta_1-\delta_1-\left(\alpha_{11}N_tX_t^2
+\alpha_{21}2N_tX_t(1-X_t)+\alpha_{31}N_t(1-X_t)^2\right)\right)\right.\\&\phantom{+}+2X_t(1-X_t)\left(\beta_2-\delta_2-\!\left(\alpha_{12}N_tX_t^2+\alpha_{22}2N_tX_t(1-X_t)+\alpha_{32}N_t(1-X_t)^2\right)\right)\\&\phantom{+}\left.+(1-X_t)^2\left(\beta_3-\delta_3-\!\left(\alpha_{13}N_tX_t^2+\alpha_{23}2N_tX_t(1-X_t)+\alpha_{33}N_t(1-X_t)^2\right)\right)\right]dt\\
dX_t&=\sqrt{\frac{\gamma X_t(1-X_t)}{N_t}}dB^2_t\\&+(1-X_t)X_t^2[(\beta_1-\delta_1)-(\beta_2-\delta_2)\\&\phantom{+(1}-\!N_t((\alpha_{11}-\alpha_{12})X_t^2
+(\alpha_{21}-\alpha_{22})2X_t(1-X_t)+(\alpha_{31}-\alpha_{32})(1-X_t)^2)]dt\\&+X_t(1-X_t)^2[(\beta_2-\delta_2)-(\beta_3-\delta_3)
\\&\phantom{+(1}-\!N_t((\alpha_{12}-\alpha_{13})X_t^2+(\alpha_{22}-\alpha_{23})2X_t(1-X_t)+(\alpha_{32}-\alpha_{33})(1-X_t)^2)]dt.
\ean 
\end{cor}

The population size and the proportion of allele $A$ are therefore directed by two independent Brownian motions. The diffusion equation \eqref{diffusionNXnoneutre} can be simplified in the neutral case:

\begin{cor}\label{corWF}
In the neutral case where $\beta_i=\beta$, $\delta_i=\delta$ and $\alpha_{ij}=\alpha$ for all $i$, $j$, the limiting diffusion $(N,X)$ introduced in Equation \eqref{diffusionNXnoneutre} satisfies:
\ban\label{diffusionNXWF}
dN_t&=\sqrt{2\gamma N_t}dB^1_t+N_t(\beta-\delta-\alpha N_t)dt\\
dX_t&=\sqrt{\frac{\gamma X_t(1-X_t)}{N_t}}dB^2_t.
\ean
$X$ is then a bounded martingale and this diffusion can be seen as a generalized Wright-Fisher diffusion (see for instance \cite{EthierKurtz} p. $411$) with a population size evolving stochastically with time.
\end{cor}

We denote by $\mathcal{C}^k_b(E,\mathbb{R})$ the set of functions from $E$ to $\mathbb{R}$ possessing bounded continuous derivatives of order up to $k$ (resp. with compact support) and $\mathcal{C}^k_c(E,\mathbb{R})$ the set of functions of $\mathcal{C}^k_b(E,\mathbb{R})$ with compact support.

\begin{proof}[Proof of Theorem \ref{TheoremConvergenceNX}] Using the Rebolledo and Aldous criteria (\cite{JoffeMetivier1986}), we prove the tightness of the sequence of processes $(\psi_1(Z^K),\psi_2(Z^K))$ and its convergence toward the unique continuous solution of a martingale problem. The proof is divided in several steps.

STEP 1. Let us denote by $L$ the generator of the diffusion process defined in Equation \eqref{diffusionA1A2}. We first prove the uniqueness of a solution $((N^A_t,N^a_t),t\in[0,T])$ to the martingale problem: for any function $f\in\mathcal{C}^2_b((\mathbb{R}_+)^2,\mathbb{R})$,
\ben \label{Mart1}M^f_t=f(N^A_t,N^a_t)-f(N^A_0,N^a_0)-\int_0^tLf(N^A_s,N^a_s)ds\een is a continuous martingale.
From \cite{StroockVaradhan}, for any $\epsilon>0$, there exists a unique (in law) solution $((N^{A,\epsilon}_t,N^{a,\epsilon}_t),t\in[0,T])$ such that for all $f\in\mathcal{C}^2_b(\mathbb{R}^2)$ the process $(M^{f,\epsilon}_t,t\in[0,T])$ such that for all $t\geq0$,
\be M^{f,\epsilon}_t=f(N^{A,\epsilon}_t,N^{a,\epsilon}_t)-f(N^{A,\epsilon}_0,N^{a,\epsilon}_0)-\int_0^tLf(N^{A,\epsilon}_s,N^{a,\epsilon}_s)\mathbf{1}_{\{\epsilon<N^{A,\epsilon}_s+N^{a,\epsilon}_s<1/\epsilon\}}ds\ee is a continuous martingale. The uniqueness of a solution of \eqref{Mart1} therefore follows from Theorem $6.2$ of \cite{EthierKurtz} about localization of martingale problems. 

STEP 2. As in the proof of Proposition \ref{propY}, we obtain easily that there exist two positive constants $C_1$ and $C_2$ such that for all $z\in(\mathbb{R}_+)^3$, the generator $L^K$ of $Z^K$, decomposed in Equation \eqref{generateurdecompose}, satisfies: $$\vert L^K\psi_1(z)\vert+\vert L^K\psi_2(z)\vert\leq C_2 \left[\phi_1(z)^2+1+K\phi_1(z)\mathbf{1}_{\phi_1(z)\geq C_1K}\right],$$ and similarly $$\vert L^K\psi_1^2(z)-2\psi_1(z)L^K\psi_1(z)+L^K\psi_2^2(z)
-2\psi_2(z)L^K\psi_2(z)\vert\leq C_2(\phi_1(z)^2+1).$$ Therefore from Proposition \ref{propgpositif}, under \eqref{hyp1} and \eqref{hyp2}, for all sequence of stopping times $\tau_K\leq T$ and for all $\epsilon>0$: \ban\label{AldousA2} \underset{K\geq K_0}{\sup}\,\underset{\sigma\leq\delta}{\sup}\,\mathbb{P}&\left(\left|\int_{\tau_K}^{\tau_K+\sigma}L^K\psi_1(N^K_s,X^K_s)ds\right|
+\left|\int_{\tau_K}^{\tau_K+\sigma}L^K\psi_2(N^K_s,X^K_s)ds\right|>\eta\right)\\&\leq\underset{K\geq K_0}{\sup}\;\mathbb{P}\left(\delta\;\underset{0\leq s\leq T+\delta}{\sup}\;C_2((N^K_s)^2+1)>\eta\right)\\&+\underset{K\geq K_0}{\sup}\;\mathbb{P}\left(\underset{0\leq s\leq T+\delta}{\sup}\; N^K_s\geq C_1 K\right)\\&\leq\epsilon\quad\text{if $K_0$ is large enough and $\delta$ is small enough.}\ean Similarly, \ban\label{AldousM2}  \underset{K\geq K_0}{\sup}\;\underset{\sigma\leq\delta}{\sup}\;\mathbb{P}&\left(\left|\int_{\tau_K}^{\tau_K+\sigma}(L^K\psi_1^2(Z^K_s)-2\psi_1(Z^K_s)L^K\psi_1(Z^K_s)
\right.\right.\\&\left.\left.\phantom{\int aaaaaaaaaaaaaaaa}+L^K\psi_2^2(Z^K_s)-2\psi_2(Z^K_s)L^K\psi_2(Z^K_s))ds\right|>\eta\right)\\&\leq\epsilon\quad\text{if $K_0$ is large enough and $\delta$ is small enough.}\ean  The sequence of processes $(\psi_1(Z^K),\psi_2(Z^K))$ is then tight from Rebolledo and Aldous criteria (Theorem $2.3.2$ of \cite{JoffeMetivier1986}). 

STEP 3. Now let us consider a subsequence of $(\psi_1(Z^K),\psi_2(Z^K))$ that converges in law in $\mathbb{D}([0,T],\mathbb{R}^2)$ toward a process $(N^A,N^a)$. Since for all $K>0$, $\underset{t\in[0,T]}{\sup}\;\Vert (N^A_t,N^a_t)-(N^A_{t^-},N^a_{t^-})\Vert\leq 2/K$ by construction, almost all trajectories of the limiting process $(N^A,N^a)$ belong to $C([0,T],\mathbb{R}^2)$.

STEP 4. Finally we prove that the sequence $\{(\psi_1(Z^K),\psi_2(Z^K))\}_{K\in\mathbb{N}^*}$ of stochastic processes converges toward the unique continuous solution of the martingale problem given by Equation \eqref{Mart1}. Indeed for every function $f\in\mathcal{C}^3_c(\mathbb{R}^2)$, from Equation \eqref{generateurdecompose} there exists a constant $C_4$ such that \ban\label{eqLKL}&\left|L^Kf(\psi_1(z),\psi_2(z))-Lf(\psi_1(z),\psi_2(z))\right|\\&\leq C_4\left[\frac{\phi_1(z)^2}{K}+\vert\phi_3(z)\vert(1+\phi_1(z))+\!\gamma K\phi_1(z)\mathbf{1}_{\phi_1(z)\geq C_2K}+\!(\phi_1(z)^2+1)\mathbf{1}_{\phi_1(z)\geq C_2K}\right]\ean Note here that the fast-scale property shown in Proposition \ref{propY}, combined to Proposition \ref{propgpositif}, will insure that $\underset{t\leq u\leq t+s}{\sup}\mathbb{E}(\vert\phi_3(Z^K_t)\vert\phi_1(Z^K_t))$ converges to $0$ when $K$ goes to infinity. Then for all $0\leq t_1<t_2<...<t_k\leq t<t+s$, for all bounded continuous measurable functions $h_1,...,h_k$ on $(\mathbb{R}_+)^2$ and every $f\in\mathcal{C}^3_c(\mathbb{R}^2)$: \ba\mathbb{E}&\left[\left(f(\psi_1(Z^K_{t+s}),\psi_2(Z^K_{t+s}))-f(\psi_1(Z^K_{t}),\psi_2(Z^K_{t}))-\int_t^{t+s}Lf(\psi_1(Z^K_u),\psi_2(Z^K_u))du\right)\right.\\&\quad\quad\quad\quad\quad\times\left.\prod_{i=1}^kh_i(\psi_1(Z^K_{t_i}),\psi_2(Z^K_{t_i}))\right]=\\\mathbb{E}&\left[\int_t^{t+s}\!\!\!\!\left(L^Kf(\psi_1(Z^K_u),\psi_2(Z^K_u))\!-\!Lf(\psi_1(Z^K_u),\psi_2(Z^K_u))\right)du\right.\\&\quad\quad\quad\quad\quad\left.\times\prod_{i=1}^kh_i(\psi_1(Z^K_{t_i}),\psi_2(Z^K_{t_i}))\right]\\&\leq \underset{i}{\sup}\|h_i\|_{\infty}\;\mathbb{E}\left[\int_t^{t+s}\left|L^Kf(\psi_1(Z^K_u),\psi_2(Z^K_u))-Lf(\psi_1(Z^K_u),\psi_2(Z^K_u))\right|du\right]\\&\leq\underset{i}{\sup}\|h_i\|_{\infty}\;s\underset{t\leq u\leq t+s}{\sup}\mathbb{E}\left[\left|L^Kf(\psi_1(Z^K_u),\psi_2(Z^K_u))-Lf(\psi_1(Z^K_u),\psi_2(Z^K_u))\right|\right]\\&\underset{K\rightarrow\infty}{\rightarrow}\!\!0,\ea  under \eqref{hyp1} and \eqref{hyp2}, from Equation \eqref{eqLKL} and Propositions \ref{propgpositif} and \ref{propY}. The extension of this result to any $f\in\mathcal{C}^2_b((\mathbb{R}_+)^2,\mathbb{R})$ is easy to obtain by approximating uniformly $f$ by a sequence of functions $f_n\in\mathcal{C}^3_c((\mathbb{R}_+)^2,\mathbb{R})$. Then from Theorem $8.10$ (p. $234$) of \cite{EthierKurtz}, $(\psi_1(Z^K),\psi_2(Z^K))$ converges in law in $\mathbb{D}([0,T],\mathbb{R}^2)$ toward the unique (in law) solution of the martingale problem given in Equation \eqref{Mart1}, which is equal to the diffusion process $(N^A,N^a)$ of Equation \eqref{diffusionA1A2}.
\end{proof}

\bigskip
The proof of Corollary \ref{corNX} relies on the following analytic lemma:

\begin{lem}\label{lemcoro}For any $x=(x^1_t,x^2_t)_{0\leq t\leq T}\in\mathbb{D}([0,T],(\mathbb{R}_+)^2)$ and any $\epsilon>0$, let us define $$\zeta_{\epsilon}(x)=\inf\{t\in[0,T]:x^1_t+x^2_t\leq2\epsilon\}.$$ Let $x=(x^1_t,x^2_t)_{0\leq t\leq T}\in\mathcal{C}([0,T],(\mathbb{R}_+)^2)$ such that $x^1_0+x^2_0>2\epsilon$ and $\epsilon'\mapsto\zeta_{\epsilon'}(x)$ is continuous in $\epsilon$. Consider a sequence of functions $(x_n)_{n\in\mathbb{Z}_+}$ such that for any $n\in\mathbb{Z}_+$, $x_n=(x^{1,n}_t,x^{2,n}_t)_{0\leq t\leq T}\in\mathbb{D}([0,T],(\mathbb{R}_+)^2)$ and $x_n$ converges to $x$ for the Skorohod topology. Then the sequence $((x^{1,n}_{t\wedge\zeta_{\epsilon}(x_n)},x^{2,n}_{t\wedge\zeta_{\epsilon}(x_n)}),t\in[0,T])$ converges to $((x^1_{t\wedge\zeta_{\epsilon}(x)},x^2_{t\wedge\zeta_{\epsilon}(x)}),t\in[0,T])$ when $n$ goes to infinity.
\end{lem}

\begin{proof} We first prove that $\zeta_{\epsilon}(x_n)$ converges to $\zeta_{\epsilon}(x)$ when $n$ goes to infinity.
For any $\delta>0$, since $\epsilon'\mapsto\zeta_{\epsilon}(x)$ is continuous in $\epsilon$, there exists $n'\in\mathbb{Z}_+^*$ such that $\zeta_{\epsilon-1/n'}(x)-\delta<\zeta_{\epsilon}(x)<\zeta_{\epsilon+1/n'}(x)+\delta$. Now let us assume that $\zeta_{\epsilon}(x_n)$ does not converge to $\zeta_{\epsilon}(x)$ when $n$ goes to infinity. Then there exists $\delta$ such that for all $n$ there exists $k_n>n$ such that $|\zeta_{\epsilon}(x_{k_n})-\zeta_{\epsilon}(x)|>\delta$. Then there exists $m$ such that
$$\underset{n\rightarrow+\infty}{\overline{\lim}}\;\underset{0\leq t\leq \zeta_{\epsilon-1/m}(x)}{\sup}\; |x^1_n(t)+x^2_n(t)-(x^1(t)+x^2(t))|\geq 1/m$$
which is impossible if $x$ is continuous. Now we prove that $(x_n)_{.\wedge\zeta_{\epsilon}(x_n)}$ converges to $x_{.\wedge\zeta_{\epsilon}(x)}$ when $n$ goes to infinity. Let us denote by $r(v,w)$ the Euclidean distance between two points $v$ and $w$ of $\mathbb{R}^2$. Since $x_n$ converges to $x$ in $\mathbb{D}([0,T],(\mathbb{R}_+)^2)$, there exists a sequence of strictly increasing functions $\lambda_n$ mapping $[0,\infty)$ onto $[0,\infty)$ such that \ben\label{xnconverge}\gamma(\lambda_n)\underset{n\rightarrow+\infty}{\longrightarrow}0 \quad\text{and}\quad\underset{n\rightarrow+\infty}{\lim}\;\underset{0\leq t\leq T}{\sup}\;r(x_n(t),x(\lambda_n(t))=0\een where $\gamma(\lambda)=\underset{0\leq t<s}{\sup}\left|\log\frac{\lambda(s)-\lambda(t)}{s-t}\right|$ (\cite{EthierKurtz}, p. $117$). Now for all $t\geq0$, \ba r(x_n(t\wedge\zeta_{\epsilon}(x_n)),x(\lambda_n(t)\wedge\zeta_{\epsilon}(x))&\leq r(x_n(t\wedge\zeta_{\epsilon}(x_n)),x(\lambda_n(t\wedge\zeta_{\epsilon}(x_n))))\\&+r(x(\lambda_n(t\wedge\zeta_{\epsilon}(x_n))),x(\lambda_n(t)\wedge\zeta_{\epsilon}(x)), \quad\text{and}\ea
\ba r(x(\lambda_n(t\!\wedge\!\zeta_{\epsilon}(x_n))),x(\lambda_n(t)\!\wedge\!\zeta_{\epsilon}(x)))\!&=\!r(x(\lambda_n(\zeta_{\epsilon}(x_n))),x(\zeta_{\epsilon}(x)))\mathbf{1}_{\{t>\zeta_{\epsilon}(x_n),\lambda_n(t)>\zeta_{\epsilon}(x)\}}\\&+\!r(x(\zeta_{\epsilon}(x)), x(\lambda_n(t)))\mathbf{1}_{\{t\leq\zeta_{\epsilon}(x_n),\lambda_n(t)>\zeta_{\epsilon}(x)\}}\\&+\!r(x(\lambda_n(\zeta_{\epsilon}(x_n))),x(\lambda_n(t)))\mathbf{1}_{\{t>\zeta_{\epsilon}(x_n),\lambda_n(t)\leq\zeta_{\epsilon}(x)\}}.\ea Therefore, using that $x$ is continuous, that $\zeta_{\epsilon}(x_n)\rightarrow\zeta_{\epsilon}(x)$ and that $\underset{0\leq t\leq T}{\sup}|\lambda_n(t)-t|\rightarrow0$ when $n$ goes to infinity, and from Equation \eqref{xnconverge}, we obtain that $\underset{n\rightarrow+\infty}{\lim}\underset{0\leq t\leq T}{\sup}\!r(x_n(t\wedge\zeta_{\epsilon}(x_n)),x(\lambda_n(t)\wedge\zeta_{\epsilon}(x))=0$ which gives the result.
\end{proof}

\begin{proof}[Proof of Corollary \ref{corNX}] Note that the function $\zeta_{\epsilon}$ defined in Lemma \ref{lemcoro} satisfies $T^K_{\epsilon}=\zeta_{\epsilon}(\psi_1(Z^K),\psi_2(Z^K))=\inf\{t\in[0,T]:N^K_t\leq\epsilon\}$, and $T_{\epsilon}=\zeta_{\epsilon}(N^A,N^a)=\inf\{t\in[0,T]:N_t\leq\epsilon\}$. From the Theorem $3.3$ of \cite{Pinsky2008}, we know that the function $\epsilon'\mapsto\zeta_{\epsilon}(N^A,N^a)$ is almost surely continuous in $\epsilon$.
Therefore from Lemma \ref{lemcoro}, the function $f$ such that for all $x\in\mathbb{D}([0,T],(\mathbb{R}_+)^2)$, $f(x)=(x_{t\wedge\zeta_{\epsilon}(x)},t\in[0,T])$ is continuous in almost all trajectories of the diffusion process $(N^A,N^a)$. Therefore from Corollary $1.9$ p.$103$ of \cite{EthierKurtz} and Theorem \ref{TheoremConvergenceNX}, if the sequence of random variables $(\psi_1(Z^K_0),\psi(Z^K_0))\in (\mathbb{R}_+)^2$ converges in law toward a random variable $(N^A_0,N^a_0)$ when $K$ goes to infinity, then for all $T>0$,the sequence of stochastic processes $(\psi_1(Z^K_{.\wedge T^K_{\epsilon}}),\psi(Z^K_{.\wedge T^K_{\epsilon}}))$ converges in law in $\mathbb{D}([0,T],(\mathbb{R}_+)^2)$ toward $(N^A_{.\wedge T_{\epsilon}},N^a_{.\wedge T_{\epsilon}})$. Since the function $(n^A,n^a)\mapsto\left(\frac{n^A+n^a}{2},\frac{n^A}{n^A+n^a}\right)$ is lipschitz continuous on $\{(n^A,n^a)\in(\mathbb{R_+})^2:n^A+n^a\geq2\epsilon\}$, we get the result.
\end{proof}

\bigskip
\begin{rem}\label{remdiplohaplo}
The diffusion process $(N_t,X_t)_{t\geq0}$ of Corollary \ref{corWF} can be compared to the haploid neutral population (which corresponds to a stochastic Lotka-Volterra process) studied in detail in \cite{CattiauxMeleard2009} and defined by:  
\ban\label{diffhaplo}
dH^1_t&=\sqrt{2\gamma H^1_t}dB^{1,h}_t+(\beta-\delta-\alpha(H^1_t+H^2_t))H^1_tdt\\
dH^2_t&=\sqrt{2\gamma H^2_t}dB^{2,h}_t+(\beta-\delta-\alpha(H^1_t+H^2_t))H^2_tdt
\ean where $B^{1,h}$ and $B^{2,h}$ are independent Brownian motions. Here, $H^1$ is the number of alleles $A$ while $H^2$ is the number of alleles $a$. $N^h=H^1+H^2$ is then the total number of individuals while $X^h=H^1/(H^1+H^2)$ is the proportion of alleles $A$ in the haploid population. We easily see that the total population size satisfies the same diffusion equation in the haploid and diploid populations. We therefore compare the stochastic processes $(N,X)$ and $(N^h,X^h)$. Now by Itô's formula, the stochastic process $(N^h,X^h)$ satisfies a diffusion equation that can be written using a new $2$-dimensional brownian motion $(\tilde{B}^1,\tilde{B}^2)$ as:
\ban\label{diffusionNXhaploid}
dN^h_t&=(\beta-\delta-\alpha N^h_t)N^h_tdt+\sqrt{2\gamma N^h_t}d\tilde{B}^1_t\\
dX^h_t&=\sqrt{\frac{2\gamma X^h_t(1-X^h_t)}{N^h_t}}d\tilde{B}^2_t
\ean
Then the differences between the haploid and the diploid neutral models only reside in a variation of the proportion of allele $A$ divided by $\sqrt{2}$ in the diploid population (see Equations \eqref{diffusionA1A2} and \eqref{diffusionNXhaploid}). However note from Equations \eqref{diffusionNXnoneutre} and \eqref{diffhaplo} that this apparently insignificant difference induce that the respective numbers of alleles $A$ and $a$ are directed by correlated Brownian motions in a diploid population which is not the case in a haploid population.
\end{rem}

\section{New change of variable and quasi-stationarity}\label{sectionQSD}

In this section we study the long-time behavior of the diffusion process $(N^A,N^a)$ introduced in Theorem \ref{TheoremConvergenceNX}. For any process $U$, we denote by $\mathbb{P}^U_x$ the distribution law of $U$ starting from a point $x$, and $\mathbb{E}^{U}_x$ the associated expectation. First, the process $N=N^A+N^a$ defined in Corollary \ref{corNX} reaches $0$ almost surely in finite time: 

\begin{prop}\label{extinction} Let $T_0=\inf\{t\geq0:N_t=0\}$. Under \eqref{hyp1}, $\mathbb{P}^N_x(T_0<+\infty)=1$ for all $x\in \mathbb{R}_+$,  and there exists $\lambda>0$ such that $\underset{x}{\sup}\;\mathbb{E}_x(e^{\lambda T_0})<+\infty$.
\end{prop}

\begin{proof}
Under \eqref{hyp1}, as in the proof of Proposition \ref{propgpositif}, there exists a positive constant $m$ such that $N$ is stochastically dominated by a diffusion process $\overline{N}$ satisfying $d\overline{N}_t=\sqrt{2\gamma \overline{N}_t} dB^1_t+\overline{N}_t(\underset{i}{\sup}\,\beta_i-\underset{i}{\inf}\,\delta_i-m\overline{N}_t)dt$. Theorem $5.2$ of \cite{CattiauxCollet...2009} gives the result for $\overline{N}$ and therefore for $N$.
\end{proof}

\bigskip The long-time behavior of the diffusion $(N^a,N^A)$ is therefore trivial and we now study the long-time behavior of this diffusion process conditioned on non-extinction, i.e. conditioned on not reaching the absorbing state $(0,0)$. In particular, we are interested in studying the possibility of a long-time coexistence of the two alleles $A$ and $a$ in the population conditioned on non-extinction.

\subsection{New change of variables}

To study the quasi-stationary behavior of the diffusion $(N,X)$ conditioned on non-extinction, we need to change variables in order to obtain a $2$-dimensional Kolmogorov diffusion (i.e. a diffusion process with a diffusion coefficient equal to $1$ and a gradient-type drift coefficient) whose quasi-stationary behavior can be most easily derived. Such ideas have been developed in \cite{CattiauxCollet...2009} and \cite{CattiauxMeleard2009}. Let us define, as long as $N_t>0$:
\ban\label{changementvariables}
S^1_t&=\sqrt{\frac{2N_t}{\gamma}}\cos\left(\frac{\arccos(2X_t-1)}{\sqrt{2}}\right)\\
S^2_t&=\sqrt{\frac{2N_t}{\gamma}}\sin\left(\frac{\arccos(2X_t-1)}{\sqrt{2}}\right).
\ean If $N_t=0$, we obviously set $S_t=(S^1_t,S^2_t)=(0,0)$. To begin with, simple calculations give the following Proposition, illustrated in Figure \ref{figureS}.

\begin{prop} \label{propespaceD}
For all $t\geq0$, $S^2_t\geq0$ and $S^2_t\geq uS^1_t$ with $u=tan\left(\frac{\pi}{\sqrt{2}}\right)<0$. \end{prop}

\medskip
\begin{proof}
For all $t\geq0$, $2X_t-1\in[-1,1]$, which gives that $\frac{\arccos(2X_t-1)}{\sqrt{2}}\in[0,\pi/\sqrt{2}]$ and $\sin\left(\frac{\arccos(2X_t-1)}{\sqrt{2}}\right)>0$. Then $S^2_t\geq0$ for all $t\geq0$. Now if $\frac{\arccos(2X_t-1)}{\sqrt{2}}\in[0,\pi/2]$, then $S^1_t\geq0$ and $S^2_t\geq0$, so $S^2_t\geq uS^1_t$. Finally, if $\frac{\arccos(2X_t-1)}{\sqrt{2}}\in]\pi/2,\pi/\sqrt{2}]$, then $S^1_t<0$, $S^2_t\geq0$, and $\frac{S^2_t}{S^1_t}=\tan\left(\frac{\arccos(2X_t-1)}{\sqrt{2}}\right)\in]-\infty,u]$. Then $S^2_t\geq uS^1_t$.
\end{proof}

\begin{rem}
Let us define for all $(s_1,s_2)\in\mathbb{R}^2$, the sets $\mathbf{A}=\{s_2=0, s_1>0\}$, $\mathbf{a}=\{s_2=us_1, s_2>0\}$ and $\mathbf{0}=\{s_1=s_2=0\}$. The sets $\{S_t\in\mathbf{A}\}$, $\{S_t\in\mathbf{a}\}$, and $\{S_t\in\mathbf{0}\}$ are respectively equal to the sets $\{X_t=1\}$ (fixation of allele $A$), $\{X_t=0\}$ (fixation of allele $a$) and $\{N_t=0\}$ (extinction of the population).
\end{rem} 

We denote by $\mathcal{D}=\mathbf{R}\times\mathbf{R}_+\cap\{(S^1,S^2):S^2\geq uS^1\}$ the set of values taken by $S_t$ for $t\geq0$, $\partial \mathcal{D}=\mathbf{A}\cup\mathbf{a}\cup\mathbf{0}$ its boundary in $\mathbb{R}^2$, and $T_{\mathbf{D}}$ the hitting time of $\mathbf{D}$ for any $\mathbf{D}\subset \mathcal{D}$. $\mathbf{0}$, $\mathbf{A}\cup\mathbf{0}$ and $\mathbf{a}\cup\mathbf{0}$ are therefore absorbing sets and from Proposition \ref{extinction}, starting from any point $s\in\mathcal{D}$, $S$ reaches any of these sets almost surely in finite time. 

\begin{figure}\begin{center}
\scalebox{0.7}{
\begin{pspicture}(-5,-1.2)(7,7)
\psline{->}(0,0)(6,0)\psline{->}(0,0)(0,5.5)\psline(0,0)(-4,5.25)
\put(6.2,-0.2){$S^1$}\put(0,5.6){$S^2$}\put(0.2,4){$\mathbf{a}_0$}
\psline(0,0)(5,3.8)\put(5.1,3.9){$\mathbf{A}_0$}\psline(0,0)(3,5)\put(3.1,5.1){$\mathbf{M}$}

\uput{0}[0]{-53}(-3.2,3.2){$\mathbf{a}=\{s_2=us_1\}$}\put(3,-0.4){$\mathbf{A}=\{s_2=0\}$}\put(-1.3,-0.4){$\mathbf{0}=\{s_1=s_2=0\}$}

\psdot[dotsize=0.1](0,0)
\psdot[dotstyle=x, dotsize=0.2](4.5,2)\psline{<->}(0,0)(4.5,2)\psarc(0,0){4.1}{0}{24}\put(4.1,0.9){$\frac{arccos(2X-1)}{\sqrt{2}}$}\put(4.6,2.1){$(s_1,s_2)=\psi(n,x)$}\put(1.9,0.4){$\sqrt{\frac{\gamma N}{2}}$}

\pstGeonode[PointName=none,PointSymbol=none](0,0){O}
\pstGeonode[PointName=none,PointSymbol=none](-4,5.25){B}
\pstGeonode[PointName=none,PointSymbol=none](5,3.8){A}
\pstRightAngle{B}{O}{A}
\end{pspicture}
}
\end{center}
\caption[Set of values taken by the diffusion $S$]{\label{figureS}Set $\mathcal{D}$ of the values taken by $S_t$, for $t\geq0$.}
\end{figure}
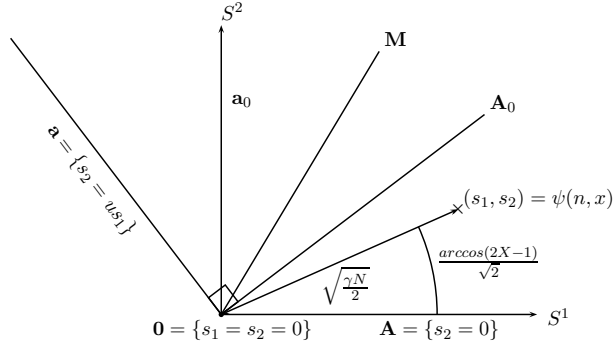 
Finally,

\begin{prop}\label{cgtvariableS}
The transformation \ba\psi:\!\mathbb{R}_+^*\!\times\![0,1]\!&\rightarrow\mathcal{D}\setminus\textbf{0}\\
(n,x)&\mapsto\!(s_1,s_2)=\!\left(\sqrt{\frac{2n}{\gamma}}\cos\!\left(\frac{\arccos(2x-1)}{\sqrt{2}}\right)\!,\sqrt{\frac{2n}{\gamma}}\sin\!\left(\frac{\arccos(2x-1)}{\sqrt{2}}\right)\!\right)\ea introduced in Equation \eqref{changementvariables} is a bijection.
\end{prop}

\begin{proof}
For any $(s_1,s_2)\in\mathcal{D}\setminus\textbf{0}$, we easily get the following inverse transformation: \be\begin{array}{ll}x=\left\{\begin{array}{l}
\frac{1+\cos\left(\sqrt{2}\arctan\left(\frac{s_2}{s_1}\right)\right)}{2}\quad\text{if $s_1\geq0$,}\\
\frac{1+\cos\left(\sqrt{2}\arctan\left(\frac{s_2}{s_1}+\pi\right)\right)}{2}\quad\text{if $s_1\leq0$,}
\end{array}\right.&\quad\text{and}\quad n=\frac{\left((s_1)^2+(s_2)^2\right)\gamma}{2}
,\end{array}\ee for which we obviously have $n\in\mathbb{R}^*_+$ and $x\in[0,1]$.
\end{proof}

\bigskip Now from Itô's formula, $S$ satisfies the following diffusion equation:\ban\label{diffusionSnonneutregenerale}
dS^1_t&=dW^1_t-q_1(S_t)dt\\
dS^2_t&=dW^2_t-q_2(S_t)dt,
\ean where, in the neutral case (Equation \eqref{neutre}), $q(s)$ is defined for all $s=(s_1,s_2)\in\mathcal{D}$ such that $s_1\geq0$ by
\ban \label{qneutre}
q(s)=\left(\begin{array}{c}
-\frac{s_2}{(s_1)^2+(s_2)^2}\frac{1}{\sqrt{2}\tan\left(\sqrt{2}\arctan\left(\frac{s_2}{s_1}\right)\right)}\\-s_1\left[\left(\beta-\delta-\frac{\alpha\gamma}{2}((s_1)^2+(s_2)^2)\right)\frac{1}{2}-\frac{1}{(s_1)^2+(s_2)^2}\right] \\  \\
\frac{s_1}{(s_1)^2+(s_2)^2}\frac{1}{\sqrt{2}\tan\left(\sqrt{2}\arctan\left(\frac{s_2}{s_1}\right)\right)}\\-s_2\left[\left(\beta-\delta-\frac{\alpha\gamma}{2}((s_1)^2+(s_2)^2)\right)\frac{1}{2}-\frac{1}{(s_1)^2+(s_2)^2}\right]
\end{array} \right)
\ean and when $s_1\leq0$ by
\ba
q(s)=\left(\begin{array}{c}
-\frac{s_2}{(s_1)^2+(s_2)^2}\frac{1}{\sqrt{2}\tan\left(\sqrt{2}\left(\arctan\left(\frac{s_2}{s_1}\right)+\pi\right)\right)}\\-s_1\left[\left(\beta-\delta-\frac{\alpha\gamma}{2}((s_1)^2+(s_2)^2)\right)\frac{1}{2}-\frac{1}{(s_1)^2+(s_2)^2}\right] \\ \\
\frac{s_1}{(s_1)^2+(s_2)^2}\frac{1}{\sqrt{2}\tan\left(\sqrt{2}\left(\arctan\left(\frac{s_2}{s_1}\right)+\pi\right)\right)}\\-s_2\left[\left(\beta-\delta-\frac{\alpha\gamma}{2}((s_1)^2+(s_2)^2)\right)\frac{1}{2}-\frac{1}{(s_1)^2+(s_2)^2}\right]
\end{array} \right).
\ea

The formula for $q$ in the general case and if $s_1\geq0$ is given in Appendix \ref{appq}. 

We now give conditions on the demographic parameters so that $S$ satisfies $dS_t=dW_t-\nabla Q(S_t)dt$, i.e. $q=(q_1,q_2)=\nabla Q$ for a real-valued function $Q$ of two variables. This requires at least that $\frac{\partial q_2}{\partial{s_1}}(s)=\frac{\partial q_1}{\partial{s_2}}(s)$ for all $s\in\mathcal{D}$. We state the following 

\begin{prop}\label{PropQnonneutre}
\begin{description}
\item[$(i)$] $\frac{\partial q_2(s)}{\partial  s_1}=\frac{\partial q_2(s)}{\partial  s_2}$ for all $s=(s_1,s_2)\in\mathcal{D}$ if and only $\alpha$ is symmetric, i.e. $\alpha_{12}=\alpha_{21}$, $\alpha_{31}=\alpha_{13}$, $\alpha_{23}=\alpha_{32}$. 
\item[$(ii)$] In this case we have \ban dS_t=dW_t-\nabla Q(S_t)dt,\label{diffkolmo}\ean 
with, in the neutral case and for all $s=(s_1,s_2)\in\mathcal{D}$,

\ben \label{equationQ} Q(s)=\left\{\!\!\begin{array}{l}
\frac{\ln((s_1)^2+(s_2)^2)}{2}+\frac{1}{2}\ln\left(\sin\left(\sqrt{2}\arctan\left(\frac{s_2}{s_1}\right)\right)\right)\\\phantom{\ln((s_1)^2+(s_2)^2)}-(\beta-\delta-\frac{\alpha\gamma}{4}((s_1)^2+(s_2)^2))\frac{(s_1)^2+(s_2)^2}{4}\;\text{if $s_1\geq0$}\\ \\
\frac{\ln((s_1)^2+(s_2)^2)}{2}+\frac{1}{2}\ln\left(\sin\left(\sqrt{2}\left(\arctan\left(\frac{s_2}{s_1}\right)+\pi\right)\right)\right)\\\phantom{\ln((s_1)^2+(s_2)^2)}-(\beta-\delta-\frac{\alpha\gamma}{4}((s_1)^2+(s_2)^2))\frac{(s_1)^2+(s_2)^2)}{4}\;\text{if $s_1\leq0$}.
\end{array}\right.\een This function $Q$ in the non-neutral case is given in Appendix \ref{appQ}.
\end{description}
\end{prop}

\begin{proof} For $(i)$,
we can decompose the functions $q_1$ and $q_2$ as:
\ban \label{formuleq1}q_1(s)&=\frac{\gamma}{2n}s_1+\frac{\gamma}{\sqrt{2}n}s_2\frac{2x-1}{4\sqrt{x(1-x)}}\\&-\frac{s_1}{2}[x^2U+2x(1-x)V+(1-x)^2W]\\&-\frac{s_2}{\sqrt{2}}\sqrt{x(1-x)}[x(U-V)+(1-x)(V-W)]
\ean and \ban \label{formuleq2}q_2(s)&=\frac{\gamma}{2n}s_2-\frac{\gamma}{\sqrt{2}n}s_1\frac{2x-1}{4\sqrt{x(1-x)}}\\&-\frac{s_2}{2}[x^2U+2x(1-x)V+(1-x)^2W]\\&+\frac{s_1}{\sqrt{2}}\sqrt{x(1-x)}[x(U-V)+(1-x)(V-W)]
\ean where 
\ban \label{formuleXNappendix}&\begin{array}{ll}x=\left\{\begin{array}{l}
\frac{1+\cos\left(\sqrt{2}\arctan\left(\frac{s_2}{s_1}\right)\right)}{2}\quad\text{if $s_1\geq0$,}\\ 
\frac{1+\cos\left(\sqrt{2}\arctan\left(\frac{s_2}{s1}+\pi\right)\right)}{2}\quad\text{if $s_1\leq0$,}
\end{array}\right.& n=\frac{\left((s_1)^2+(s_2)^2\right)\gamma}{2},
\end{array}\\ & U=\beta_1-\delta_1-n\left(\alpha_{11}x^2+\alpha_{21}2x(1-x)+\alpha_{31}(1-x)^2\right)\\ & V=\beta_2-\delta_2-n\left(\alpha_{12}x^2+\alpha_{22}2x(1-x)+\alpha_{32}(1-x)^2\right), \text{ and }\\& W=\beta_3-\delta_3-n\left(\alpha_{13}x^2+\alpha_{23}2x(1-x)+\alpha_{33}(1-x)^2\right).\ean From Equation \eqref{formuleXNappendix}, we easily obtain that: $$\frac{\partial n(s_1,s_2)}{\partial s_1}=\frac{s_1}{s_2}\frac{\partial n(s_1,s_2)}{\partial s_2}\quad\text{and}\quad\frac{\partial x(s_1,s_2)}{\partial s_1}=-\frac{s_2}{s_1}\frac{\partial x(s_1,s_2)}{\partial s_2}.$$ Finally, after some calculations and using that $$\frac{\partial n}{\partial s_1}=\gamma s_1\quad\text{and}\quad\frac{\partial x}{\partial s_1}\times\left(\frac{(s_1)^2+(s_2)^2}{s_2}\right)=\sqrt{2x(1-x)},$$ we obtain that $\frac{\partial q_1(s)}{\partial s_2}=\frac{\partial q_2(s)}{\partial s_1}$ if and only if for all $x\in[0,1]$, \ba x^2[\alpha_{21}\!-\alpha_{31}\!-\alpha_{12}\!+\alpha_{13}+\alpha_{32}-\alpha_{23}]
+x[\alpha_{31}\!-\alpha_{13}\!+2\alpha_{23}\!-2\alpha_{32}]
+[\alpha_{32}\!-\alpha_{23}]=0\ea which happens if and only if $\alpha$ is symmetric. For $(ii)$, the result comes from straightforward calculations that are given in the general case in Appendix \ref{appQ}.
\end{proof}

\bigskip Assuming now that $\alpha$ is symmetric, we can establish some sufficient conditions on the parameters $\alpha_{ij}$ so that the function $g$ introduced in Proposition \ref{propgpositif} is positive, i.e. Hypothesis \eqref{hyp1} is satisfied.
\begin{prop} \label{propconditionsalpha} Let us now assume that $\alpha_{ij}=\alpha_{ji}$ for all $i,j\in\{1,2,3\}$. If $\alpha_{ii}>0$ for all $i\in\{1,2,3\}$ and one of the following conditions is satisfied:
\begin{description}
\item[$(i)$] $\alpha_{ij}>0$ for all $i,j$. 
\item[$(ii)$] There exists $i\in\{1,2,3\}$ such that $\alpha_{ik}>0$ for all $k$, and $\alpha_{jl}^2<\alpha_{jj}\alpha_{ll}$ if $i$, $j$ and $l$ are all distinct.
\item[$(iii)$] There exists $i\in\{1,2,3\}$ such that $\alpha_{ii}\alpha_{jl}>\alpha_{ij}\alpha_{il}$,  $\alpha_{ij}^2<\alpha_{ii}\alpha_{jj}$, and $\alpha_{il}^2<\alpha_{ii}\alpha_{ll}$ where $i$, $j$ and $l$ are all distinct.
\item[$(iv)$] There exists $i\in\{1,2,3\}$ such that $\alpha_{ij}^2<\alpha_{ii}\alpha_{jj}$, $\alpha_{il}^2<\alpha_{ii}\alpha_{ll}$, and $(\alpha_{ii}\alpha_{jl}-\alpha_{ij}\alpha_{il})^2<(\alpha_{ii}\alpha_{ll}-\alpha_{il}^2)(\alpha_{ii}\alpha_{jj}-\alpha_{ij}^2)$ where $i$, $j$ and $l$ are all distinct.
\end{description} then Hypothesis \eqref{hyp1} is satisfied.
\end{prop}

\begin{proof} Since $\alpha$ is symmetric, we have for all $z=(z_1,z_2,z_3)\in(\mathbb{R}_+)^3$: $$g(Z)=\alpha_{11}(z_1)^2+2\alpha_{12}z_1z_2+\alpha_{22}(z_2)^2+2\alpha_{23}z_2z_3+2\alpha_{13}z_1z_3+\alpha_{33}(z_3)^2.$$ Considering $g$ as a polynomial function of $z_1$, we easily obtain that $g$ is positive if $(1):$ the discriminant $\Delta_1(z_2,z_3)=(2\alpha_{12}z_2+2\alpha_{13}z_3)^2-4\alpha_{11}(\alpha_{22}(z_2)^2+\alpha_{33}(z_3)^2+2\alpha_{23}z_2z_3)$ is negative or if $(2):$ $(2\alpha_{12}z_2+2\alpha_{13}z_3)>\sqrt{\Delta_1(z_2,z_3)}$. If $\alpha_{12}>0$, $\alpha_{13}>0$, and $\alpha_{23}>0$ or  $\alpha_{23}^2<\alpha_{22}\alpha_{33}$ (case $(i)$ or $(ii)$), then $(2)$ is true for all $z\in\left(\mathbb{R}_+\right)^3$. If $\alpha_{11}\alpha_{22}>\alpha_{12}^2$, $\alpha_{11}\alpha_{33}>\alpha_{13}^2$, and $\alpha_{11}\alpha_{23} >\alpha_{12}\alpha_{13}$ or $(\alpha_{11}\alpha_{23}-\alpha_{12}\alpha_{13})^2<(\alpha_{11}\alpha_{33}-\alpha_{13}^2)(\alpha_{11}\alpha_{22}-\alpha_{12}^2)$ (case $(iii)$ or $(iv)$), then $(1)$ is true for all $z\in\left(\mathbb{R}_+\right)^3$, which gives the result, allowing in the end for permutations of indices $1$, $2$, and $3$. \end{proof}

\bigskip Note that these conditions mean that for Hypothesis \eqref{hyp1} to be true, we need that cooperation is not too strong or is compensated in some way by competition.

\subsection{Absorption of the diffusion process $S$}

In this section, we establish more precise results concerning the absorption of the process $S$ in the absorbing sets $\mathbf{0}$, $\mathbf{A}\cup\mathbf{0}$, $\mathbf{a}\cup\mathbf{0}$ and $\mathbf{A}\cup\mathbf{a}\cup\mathbf{0}$. 

\begin{thm}\label{theoremtemps}
\begin{description}
\item[$(i)$] For all $s\in \mathcal{D}\setminus\mathbf{0}$, $\mathbb{P}^S_s(T_{\mathbf{A}}\wedge T_{\mathbf{a}}<T_{\mathbf{0}})=1$.
\item[$(ii)$] Let $\mathbf{M}=\{(s_1,s_2)\in\mathcal{D}: s_2=tan(\frac{\pi}{2\sqrt{2}}\}$ (see Figure \ref{figureS}). For all $s\in\mathbf{M}$, $\mathbb{P}^S_s(T_{\mathbf{a}}<T_{\mathbf{0}})>0$ and  $\mathbb{P}^S_s(T_{\mathbf{A}}<T_{\mathbf{0}})>0$.
\item[$(iii)$] For all $s\in \mathcal{D}\setminus\partial \mathcal{D}$, $\mathbb{P}^S_s(T_{\mathbf{A}}<T_{\mathbf{0}})>0$, and $\mathbb{P}^S_s(T_{\mathbf{a}}<T_{\mathbf{0}})>0$.
\end{description}
\end{thm}

\begin{proof} We first consider the neutral case.  To prove $(i)$, we start with extending Girsanov approach as presented in \cite{CattiauxMeleard2009} (proof of Proposition $2.3$), on two different subsets of $\mathcal{D}$. Let us indeed define (see Figure \ref{figureS}) \ba \mathcal{D}_1&=\{(s_1,s_2)\in\mathcal{D}, s_1\geq0\}=(\mathbb{R}_+)^2,\\  \mathcal{D}_2&=\{(s_1,s_2)\in\mathcal{D}, us_2-s_1\geq0\},\\ \mathbf{A}_0&=\{(s_1,s_2)\in\mathcal{D}, s_1=-us_2\},\quad\text{ and}\\ \mathbf{a}_0&=\{(s_1,s_2)\in\mathcal{D}, s_1=0\}.\ea Note that $\partial\mathcal{D}_1=\mathbf{A}\cup\mathbf{a_0}\cup\mathbf{0}$ and that $\partial\mathcal{D}_2=\mathbf{a}\cup\mathbf{A_0}\cup\mathbf{0}$). Let us first assume that $S$ starts in $\mathcal{D}_1$. We consider the diffusion process $H$ which is solution of the following stochastic differential equation: \ba dH^1_t&=dB^1_t-\frac{1}{2H^1_t}dt\\
dH^2_t&=dB^2_t.\ea Then $H^1$ and $H^2$ are independent diffusion processes defined up to their respective hitting time of origin $T^{H^1}_0$ and $T^{H^2}_0$. Let us define for all $x\in\mathbb{R}_+^*$, $Q_1(x)=\frac{\ln(x)}{2}.$ Then from Girsanov Theorem extension, for $i\in\{1,2\}$, for all $t>0$, for all bounded Borel function $f$ on $C([0,t],\mathbb{R}_+)$ and for all $x\in\mathbb{R}_+^*$, \ben\label{GirsanovHW}\mathbb{E}^{H^i}_x\left(f(w)\mathbf{1}_{t<T_0(w)}\right)=\mathbb{E}^{W}_x
\left(f(w)\mathbf{1}_{t<T_0(w)}e^{G_i(t)}\right),\een where $W$ is a $1-$dimensional brownian motion and \ba G_1(t)&=Q_1(x)-Q_1(w_t)-\frac{1}{2}\int_0^t\left((Q_1'(w_u))^2-Q_1''(w_u)\right)du,\\G_2(t)&=0.\ea Therefore the law of the couple of stopping times $(T^{H^1}_0,T^{H^2}_0)$ is equivalent to the Lebesgue measure on $]0,\infty[\otimes]0,\infty[$. We now consider the diffusion processes $S$ and $H$ starting from $s\in\mathcal{D}_1$ and stopped when they reach $\partial\mathcal{D}_1=\mathbf{A}\cup\mathbf{a}_0\cup \mathbf{0}$ and define for all $(x_1,x_2)\in(\mathbb{R}_+)^2$, $Q_2(x_1,x_2)=Q_1(x_1)-Q(x_1,x_2)$ (where $Q$ is given in the neutral case in Equation \eqref{equationQ} and in the general case in Equation \eqref{equationQnonneutre}). Then from the extended Girsanov theory again, we have for all bounded Borel function $f$ on $C([0,t],(\mathbb{R}_+)^2)$ and for all $s\in\mathcal{D}_1\setminus\mathbf{0}$, $$\mathbb{E}^{S}_s\left(f(w)\mathbf{1}_{t<T_{\partial \mathcal{D}_1}(w)}\right)=\mathbb{E}^{H}_s
\left(f(w)\mathbf{1}_{t<T_{\partial \mathcal{D}_1}(w)}e^{R_t}\right)$$ where $R_t=Q_2(s)-Q_2(w_t)-\frac{1}{2}\int_0^t\left(\left|\nabla Q_2(w_u)\right|^2-\Delta Q_2(w_u)\right)du.$ Now $R_{t\wedge T_{\mathbf{A}}\wedge T_{\mathbf{a}_0}}$ is well defined, which gives for all $s\in\mathcal{D}_1$: \be\mathbb{E}^S_s\left[f(w)\mathbf{1}_{t<T_{\mathbf{0}}}(w)\right]
=\mathbb{E}^{H}_s\left[f(\omega)\mathbf{1}_{t<T_{\mathbf{0}}(\omega)}\exp(R_{t\wedge T_{\mathbf{A}}\wedge T_{\mathbf{a}_0}})\right]\ee Then from Equation \eqref{GirsanovHW}, for any $s\in\mathcal{D}_1$,\ben\label{eqAB0}\mathbb{P}^S_s(T_{\mathbf{A}}\wedge T_{\mathbf{a}_0}<T_{\mathbf{0}})=1.\een Now, in the neutral case, the proportion $1-X$ of allele $a$ is a bounded martingale, from Corollary \ref{corNX}, which gives that starting from any $s\in\mathbf{A}_0\subset\mathcal{D}_1$, \ba\mathbf{E}^S_{s}\left(1-X_{T_{\mathbf{A}}\wedge T_{\mathbf{a}_0}}\right)&=\mathbb{P}^S_s(T_{\mathbf{A}}> T_{\mathbf{a}_0})\times\frac{1-\cos(\frac{\pi}{\sqrt{2}})}{2}
=\mathbb{E}^S_s(1-X_0)=\frac{1+\cos(\frac{\pi}{\sqrt{2}})}{2}.\ea Finally, the same work can be done on $\mathcal{D}_2$ by symmetry, which gives that for all $s\in\mathcal{D}_2$, \ben\label{eqBA0}\mathbb{P}^S_s(T_{\mathbf{a}}\wedge T_{\mathbf{A}_0}<T_{\mathbf{0}})=1,\een and for all $s\in\mathbf{a}_0\subset\mathcal{D}_2$,
\be\mathbb{P}^S_s(T_{\mathbf{a}}> T_{\mathbf{A}_0})=\frac{1+\cos(\frac{\pi}{\sqrt{2}})}{1-\cos(\frac{\pi}{\sqrt{2}})}.\ee Then the number of back and forths of $S$ between $\mathbf{A}_0$ and $\mathbf{a}_0$ follows a geometrical law with parameter $\frac{1+\cos(\frac{\pi}{\sqrt{2}})}{1-\cos(\frac{\pi}{\sqrt{2}})}$ and is therefore almost surely finite. What is more, from Equations \eqref{eqAB0} and \eqref{eqBA0}, each time the diffusion $S$ reaches $\mathbf{A}_0$ (resp. $\mathbf{a}_0$), it goes to $\mathbf{a}_0$ or $\mathbf{A}$ (resp. $\mathbf{A}_0$ or $\mathbf{a}$) before $\mathbf{0}$ almost surely, which gives the result for the neutral case. Now note from Equation \eqref{changementvariables} that $S_t\in\mathbf{M}$ if and only $X_t=1/2$. Therefore $(ii)$ is obvious in the neutral case since by symmetry, for all $s\in\mathbf{M}$, $\mathbb{P}^S_s(T_{\mathbf{a}}<T_{\mathbf{A}})=1/2$ which gives that $\mathbb{P}^S_s(T_{\mathbf{a}}<T_{\mathbf{0}})=\mathbb{P}^S_s(T_{\mathbf{A}}<T_{\mathbf{0}})=1/2>0$ from $(i)$. Finally for $(iii)$, using Girsanov theory as in the proof of $(i)$, for all $s\in\mathcal{D}\setminus\partial\mathcal{D}$, $\mathbb{P}^S_s(T_{\mathbf{M}}<T_{\mathbf{0}})=\mathbb{P}^S_s(T_{\mathbf{M}}<\infty)>0$. Now by Markov's Property, for all $s\in\mathcal{D}\setminus\partial\mathcal{D}$, we get  \ba\mathbb{P}^S_s(T_{\mathbf{a}}<T_{\mathbf{0}})&\geq\mathbb{P}^S_s(T_{\mathbf{a}}<T_{\mathbf{0}}, T_{\mathbf{M}}<\infty)=\frac{\mathbb{P}^S_s(T_{\mathbf{M}}<\infty)}{2}\quad\text{from $(i)$}.\ea Similarly, for all $s\in\mathcal{D}\setminus\partial\mathcal{D}$, $\mathbb{P}^S_s(T_{\mathbf{A}}<T_{\mathbf{0}})>0$. 

In the non-neutral case, by Girsanov theory again, the law of the process $(S^1,S^2)$ starting from $(s_1,s_2)\in\mathcal{D}$ is equivalent on $C([0,t],\mathcal{D})$ to the law of a process $(\tilde{S}^1,\tilde{S}^2)$ starting from $(s_1,s_2)$ and that is neutral, which gives all the results.
\end{proof}

\subsection{Quasi-stationary behavior of $S$}

In \cite{CattiauxMeleard2009}, the study of quasi-stationary distributions has been developped for diffusion processes of the form \eqref{diffkolmo}. In particular, existence and uniqueness is given under some conditions on the diffusion coefficient $Q$. Let us prove that these conditions are satisfied in our case.

\begin{prop}\label{PropQnorme} 
\begin{description}
\item[$(i)$] There exists a constant $C$ such that for all $s=(s_1,s_2)\in \mathcal{D}$,
$$\vert\nabla Q(s)\vert^2-\Delta Q(s)\geq C.$$
\item[$(ii)$] $\inf\{\vert\nabla Q(s)\vert^2-\Delta Q(s), |s|\geq R, s\in\mathcal{D}\}\rightarrow +\infty$ when $R\rightarrow\infty$.
\end{description}
\end{prop}

\begin{proof} Let us define $F(s)=\vert\nabla Q(s)\vert^2-\Delta Q(s)$ for all $s\in \mathcal{D}$. In the neutral case, we find: \ba F(s)&=\frac{1}{2((s_1)^2+(s_2)^2)\tan^2\left(\sqrt{2}\arctan\frac{s_2}{s_1}\right)}\\&+((s_1)^2+(s_2)^2)\left[\frac{(\beta-\delta-\frac{\alpha\gamma}{2}((s_1)^2+(s_2)^2)^2}{4}+\frac{1}{((s_1)^2+(s_2)^2)^2}\right]\\&
-((s_1)^2+(s_2)^2)\frac{\alpha\gamma}{2}\\&
+\frac{1}{(s_1)^2+(s_2)^2}\frac{1+\tan^2\left(\sqrt{2}\arctan\frac{s_2}{s_1}\right)}{\tan^2\left(\sqrt{2}\arctan\frac{s_2}{s_1}\right)}\\&
\geq C \quad\quad\text{clearly}.\ea We also have $$F(s)\geq ((s_1)^2+(s_2)^2)\left[\frac{(\beta-\delta-\frac{\alpha\gamma}{2}((s_1)^2+(s_2)^2)^2}{4}+\frac{1}{(s_1)^2+(s_2)^2}-\frac{\alpha\gamma}{2}\right]
$$ which gives $(ii)$.
The proof of the two points in the non-neutral case is given in Appendix \ref{appprop}.
\end{proof}

As in \cite{CattiauxMeleard2009}, the quasi-stationary behavior of $S$ is first studied respectively to the absorbing set $\partial \mathcal{D}$ and then for the absorbing set $\mathbf{0}$ that corresponds to the extinction of the population. 

\begin{thm}\label{theoQSD1}
\begin{description}
\item[$(i)$] There exists a unique distribution $\nu$ on $\mathcal{D}\setminus\partial \mathcal{D}$ such that for all $E\subset\mathcal{D}\setminus\partial \mathcal{D}$ and all $t\geq0$, $$\mathbb{P}^S_{\nu}(S_t\in E|T_{\partial \mathcal{D}}>t)=\nu(E).$$ What is more, this distribution is a Yaglom limit for $S$, i.e. for all $s\in\mathcal{D}\setminus\partial \mathcal{D}$, $$\underset{t\rightarrow\infty}{\lim}\mathbb{P}^S_s(S_t\in E|T_{\partial \mathcal{D}}>t)=\nu(E).$$
\item[$(ii)$]There exists a unique probability measure $\nu_0$ on $\mathcal{D}\setminus\mathbf{0}$ such that for all $s\in\mathcal{D}\setminus\partial \mathcal{D}$ and for all $E\subset\mathcal{D}\setminus\mathbf{0}$, \ben\label{existqsd}\underset{t\rightarrow\infty}{\lim}\mathbb{P}^S_s(S_t\in E\vert T_{\mathbf{0}}>t)=\nu_0(E).\een
\end{description}
\end{thm}

\begin{proof} The set of assumptions $(H)$ of \cite{CattiauxMeleard2009} (p. $816-818$) is satisfied from Propositions \ref{extinction} and \ref{PropQnorme}, which gives $(i)$ from \cite{CattiauxMeleard2009} (Proposition $B.12$). $(ii)$ is obtained as in \cite{CattiauxMeleard2009} by using Theorem \ref{theoremtemps} and decomposing: \ben\label{decompofinal} \mathbb{P}_s(S_t\in E\vert T_{\mathbf{0}}>t)=\frac{\mathbb{P}_s(S_t\in E)}{\mathbb{P}_s(T_{\partial \mathcal{D}}>t)}\frac{\mathbb{P}_s(T_{\partial \mathcal{D}}>t)}{\mathbb{P}_s(T_{\mathbf{0}}>t)}.\een\end{proof}

\bigskip 
Note that the quasi-stationary behavior of the diffusion process $((N_t,X_t),t\geq0)$ conditioned on non extinction is obtained easily since $$\mathbb{P}^{N,X}_{(n,x)}((N_t,X_t)\in F\vert N_t>0)=\mathbb{P}^S_s(S_t\in E\vert T_\mathbf{0}>t),$$ if $s=\psi(n,x)$ and $E=\psi(F)$ where $\psi$ is defined in Proposition \ref{cgtvariableS}. Let us remind that we are interested in studying the possibility of a long-time coexistence of the two alleles $A$ and $a$ in the population conditioned on non-extinction. This means that we would like to approximate the quasi-stationary distribution $\nu_X$ such that \ben\label{distribfigures}\nu_X(.):=\underset{t\rightarrow\infty}{\lim}\mathbb{P}^{N,X}_{(n,x)}(X_t\in .\vert N_t>0)\een and we are interested in knowing whether $\nu_X(]0,1[)=0$ or not. Indeed, if $\nu_X(]0,1[)\neq0$ we can observe a long-time coexistence of the two alleles in the population conditioned on non-extinction whereas if $\nu_X(]0,1[)=0$, no such coexistence is possible. Note that $\nu_X(]0,1[)=\nu_0(\mathcal{D}\setminus\partial\mathcal{D})$. For a haploid population with clonal reproduction, \cite{CattiauxMeleard2009} proved that in a pure competition case, i.e. when every individual competes with every other one, no coexistence of alleles is possible. However, in our diploid population, this result should not be true anymore. Indeed, from Equation \eqref{decompofinal}, $$\mathbb{P}_s(S_t\in\mathcal{D}\setminus\partial\mathcal{D})=\frac{\mathbb{P}_s(T_{\partial\mathcal{D}}>t)}{\mathbb{P}_s(T_{\mathbf{0}}>t)},$$ therefore the possibility of coexistence of the two alleles relies on the fact that the time spent by the population in $\mathcal{D}\setminus\partial\mathcal{D}$ is not negligible compared to the time spent in $\mathcal{D}\setminus\mathbf{0}$. In a diploid population, if the heterozygotes are favored compared to homozygous individuals (this situation is called overdominance), they can make the coexistence period last longer than the remaining lifetime of the population once one of the alleles has disappeared. Similarly, as in \cite{CattiauxMeleard2009}, cooperation can favor the long-time coexistence of alleles in the population conditioned on non-extinction. These biological and mathematical intuitions are now sustained by numerical results. 

\section{Numerical results}\label{sectionnumerical}

Numerical simulations of $\nu_X$ are obtained following the Fleming-Viot algorithm introduced in \cite{Burdzyetal1996} and which has been extensively studied in the articles \cite{Villemonais2011} and \cite{Villemonais2012}. This approach consists in approximating the conditioned distribution $$\mathbb{P}^{N,X}_{(n,x)}((N_t,X_t)\in .\vert T_0>t)$$ by the empirical distribution of an interacting particle system. More precisely, we consider a large number $k$ of particles, that all start from a given $(n,x)\in\mathbb{R}_+^*\times]0,1[$ and evolve independently from $(n,x)$ and from each other according to the law of the diffusion process $(N,X)$ defined by the diffusion equation \eqref{diffusionNXnoneutre}, until one of them hits ${N=0}$. At that time, the absorbed particle jumps to the position of one of the remaining $k-1$ particles, chosen uniformly at random among them. Then the particles evolve independently according to the law of the diffusion process $(N,X)$ until one of them reaches ${N=0}$, and so on. Theorem $1$ of \cite{Villemonais2012} gives the convergence when $k$ goes to infinity of the empirical distribution of the $k$ particles at time $t$ toward the conditioned distribution $\mathbb{P}^{N,X}_{(n,x)}((N_t,X_t)\in .\vert T_0>t)$. Here we present three biologically relevant examples. For each case, we set $k=2000$ and plot the empirical distribution at a large enough time $T$ of the $2000$ proportions of allele $A$ given by the respective positions of the $2000$ particles, starting from $(n,x)=(10,1/2)$. First, we consider a neutral competitive case, in which each individual is in competition with every other one, independently from their genotypes. Here, the quasi-stationary distribution $\nu_X$ of the proportion $X$ is a sum of two Dirac functions in $0$ and $1$ (Figure \ref{figureQSDneutre}), i.e. alleles $A$ and $a$ do not coexist in a long time limit. 

\begin{figure}[ht]\center
 \scalebox{0.25}{\includegraphics[trim=0cm 0cm 0cm 0cm,clip]{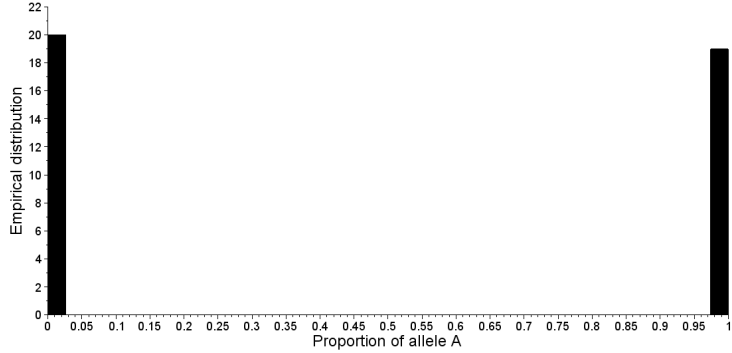}}
 \caption[Quasi-stationary distribution in a neutral competitive case]{Approximation of the quasi-stationary distribution $\nu_X$ of the proportion $X$ of allele $A$ (Equation \eqref{distribfigures}), in a neutral competitive case. In this figure, $\beta_i=1$, $\delta_i=0$, and $\alpha_{ij}=0.1$ for all $i$, $j$, and $T=40$.}\label{figureQSDneutre}
\end{figure}

Second (Figure \ref{figurecompetover}), we show an overdominance case: every individual competes equally with every other ones but heterozygous individuals are favored compared to homozygotes, as their reproduction rate is higher. In this case, the quasi-stationary distribution $\nu_X$ charges only points of $]0,1[$, i.e. alleles $A$ and $a$ seem to coexist with probability $1$. This behavior is specific to the Mendelian reproduction: in \cite{CattiauxMeleard2009}, the authors proved that no coexistence of alleles is possible in a haploid population with clonal reproduction, if every individual is in competition with every other one. 

\begin{figure}[ht]\center
\scalebox{0.25}{\includegraphics[trim=0cm 0cm 0cm 0cm,clip]{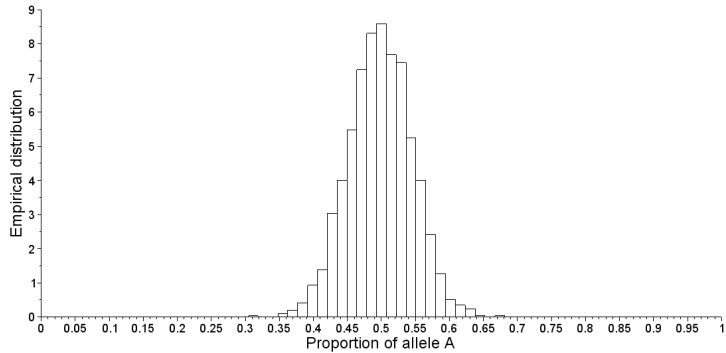}}
\caption[Quasi-stationary distribution in an overdominance case]{Approximation of the quasi-stationary distribution $\nu_X$ of the proportion $X$ of allele $A$ (Equation \eqref{distribfigures}), in an overdominance case. In this figure, $\beta_i=1$ for all $i\neq2$, $\beta_2=5$, $\delta_i=0$ for all $i$, $\alpha_{ij}=0.1$ for all $(i,j)$, and $T=100$.}\label{figurecompetover}
\end{figure}

Third (Figure \ref{figureQSDniches}), we show a case in which individuals only compete with individuals with same genotype; this can happen if different genotypes feed differently and have different predators. In this case, we can observe either a coexistence of the two alleles $A$ and $a$ or an elimination of one of the alleles, since the distribution $\nu_X$ charges both $\{0\}\cup\{1\}$ and $]0,1[$. 

\begin{figure}[ht]\center
\scalebox{0.26}{\includegraphics[trim=0cm 0cm 0cm 0cm,clip]{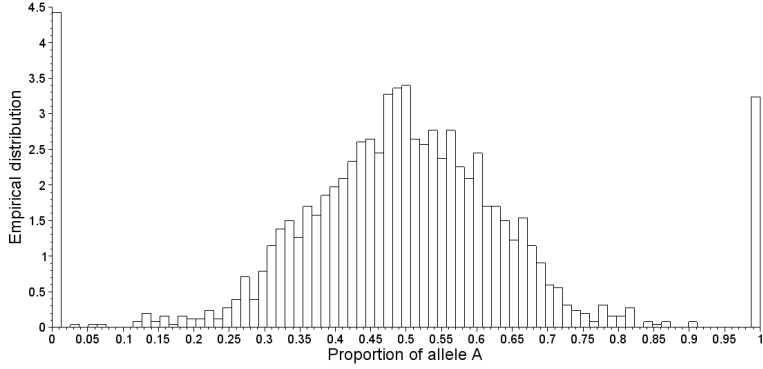}}
\caption[Quasi-stationary distribution in a separate niches case]{Approximation of the quasi-stationary distribution $\nu_X$ of the proportion $X$ of allele $A$ (Equation \eqref{distribfigures}), in a case where individuals with different genotypes do not compete or cooperate with each other. In this figure, $\beta_i=1$, $\delta_i=0$, $\alpha_{ii}=0.1$ for all $i$, $\alpha_{ij}=0$ for all $i\neq j$, and $T=2500$.}\label{figureQSDniches}
\end{figure}

\appendix

\section{Calculations in the general case}\label{appendixQnonneutre}

\subsection{Form of the function $Q$}\label{appQ}

If $\alpha$ is symmetric, we use Equations \eqref{formuleq1}, \eqref{formuleq2} and \eqref{formuleXNappendix} and search a function $Q$ such that $\frac{\partial Q(s)}{\partial s_1}=q_1(S)$ and $\frac{\partial Q(s)}{\partial s_2}=q_2(S)$. After calculating the partial derivatives of functions of the form: $$(s_1,s_2)\mapsto\left\{\begin{array}{l}
((s_1)^2+(s_2)^2)^k\cos^l\left(\sqrt{2}\arctan\left(\frac{s_2}{s_1}\right)\right)\quad\text{if $s_1\geq0$}\\ 
((s_1)^2+(s_2)^2)^k\cos^l\left(\sqrt{2}\arctan\left(\frac{s_2}{s_1}+\pi\right)\right)\quad\text{if $s_1\leq0$}
\end{array}\right.$$ for $k\in\{1,2\}$ and $l\in\{1,2,3,4\}$, we find that \ben \label{equationQnonneutre}Q(s)=\left\{\begin{array}{l}
\frac{\ln((s_1)^2+(s_2)^2)}{2}+\frac{1}{2}\ln\left(\sin\left(\sqrt{2}\arctan\left(\frac{s_2}{s_1}\right)\right)\right)\\-\frac{(s_1)^2+(s_2)^2}{4}\left[\frac{\beta_1-\delta_1+2(\beta_2-\delta_2)+\beta_3-\delta_3}{4}\right.\\
\phantom{-\frac{(s_1)^2+(s_2)^2}{4}}\left.-\frac{(s_1)^2+(s_2)^2}{4}\gamma\frac{\alpha_{11}+4\alpha_{12}+2\alpha_{13}
+4\alpha_{23}+4\alpha_{22}+\alpha_{33}}{16}\right]\\
-((s_1)^2+(s_2)^2)\left[h(s)\frac{\beta_1-\delta_1-(\beta_3-\delta_3)}{8}+h(s)^2\frac{\beta_1-\delta_1-2(\beta_2-\delta_2)+\beta_3-\delta_3}{16}\right]\\
+\frac{((s_1)^2+(s_2)^2)^2}{16}\gamma h(s)\left[\frac{\alpha_{11}+2\alpha_{12}-2\alpha_{23}-\alpha_{33}}{4}+h(s)\frac{3\alpha_{11}-2\alpha_{13}-4\alpha_{22}+3\alpha_{33}}{8}\right.\\\phantom{\frac{((s_1)^2}{16}\gamma[}\left.+h(s)^2\frac{\alpha_{11}-2\alpha_{12}+2\alpha_{23}-\alpha_{33}}{4}+h(s)^3\frac{\alpha_{11}-4\alpha_{12}
+2\alpha_{13}-4\alpha_{23}+4\alpha_{22}+\alpha_{33}}{16}\right]\\
\quad\text{if $s_1\geq0$}\\ \\ 
\frac{\ln((s_1)^2+(s_2)^2)}{2}+\frac{1}{2}\ln\left(\sin\left(\sqrt{2}\left(\arctan\left(\frac{s_2}{s_1}\right)+\pi\right)\right)\right)\\-\frac{(s_1)^2+(s_2)^2}{4}\left[\frac{\beta_1-\delta_1+2(\beta_2-\delta_2)+\beta_3-\delta_3}{4}\right.\\
\phantom{-\frac{(s_1)^2+(s_2)^2}{4}}\left.-\frac{(s_1)^2+(s_2)^2}{4}\gamma\frac{\alpha_{11}+4\alpha_{12}+2\alpha_{13}
+4\alpha_{23}+4\alpha_{22}+\alpha_{33}}{16}\right]\\
-((s_1)^2+(s_2)^2)\left[h(s)\frac{\beta_1-\delta_1-(\beta_3-\delta_3)}{8}+h(s)^2\frac{\beta_1-\delta_1-2(\beta_2-\delta_2)+\beta_3-\delta_3}{16}\right]\\
+\frac{((s_1)^2+(s_2)^2)^2}{16}\gamma h(s)\left[\frac{\alpha_{11}+2\alpha_{12}-2\alpha_{23}-\alpha_{33}}{4}+h(s)\frac{3\alpha_{11}-2\alpha_{13}-4\alpha_{22}+3\alpha_{33}}{8}\right.\\\phantom{\frac{((s_1)^2}{16}\gamma[}\left.+h(s)^2\frac{\alpha_{11}-2\alpha_{12}+2\alpha_{23}-\alpha_{33}}{4}+h(s)^3\frac{\alpha_{11}-4\alpha_{12}
+2\alpha_{13}-4\alpha_{23}+4\alpha_{22}+\alpha_{33}}{16}\right]\\
\quad\text{if $s_1\leq0$}
\end{array}\right.\een where $$h(s)=\left\{\begin{array}{l}\cos\left(\sqrt{2}\arctan\left(\frac{s_2}{s_1}\right)\right) \quad \text{when $s_1\geq0$}\\
\cos\left(\sqrt{2}\left(\arctan\left(\frac{s_2}{s_1}\right)+\pi\right)\right) \quad \text{when $s_1\leq0$.}\end{array}\right.$$

\subsection{Form of the function $q$}\label{appq}

Therefore if $s_1\geq0$: \ban\label{formuleq1app} q_1(s)&=\frac{s_1}{(s_1)^2+(s_2)^2}-\frac{s_2}{(s_1)^2+(s_2)^2}\frac{1}{\sqrt{2}\tan\left(\sqrt{2}\arctan\left(\frac{s_2}{s_1}\right)\right)}\\&-s_1\left[\frac{\beta_1-\delta_1+2(\beta_2-\delta_2)+\beta_3-\delta_3}{8}\right.\\&\left.\quad-\frac{(s_1)^2+(s_2)^2}{4}\gamma\frac{\alpha_{11}+4\alpha_{12}+2\alpha_{13}+4\alpha_{23}+4\alpha_{22}+\alpha_{33}}{16}\right]\\&-2s_1\left[h(s)\frac{\beta_1-\delta_1-(\beta_3-\delta_3)}{8}+h(s)^2\frac{\beta_1-\delta_1-2(\beta_2-\delta_2)+\beta_3-\delta_3}{16}\right]\\&+s_1\frac{((s_1)^2+(s_2)^2)}{4}\gamma h(s)\left[\frac{\alpha_{11}+2\alpha_{12}-2\alpha_{23}-\alpha_{33}}{4}\right.\\&\quad+h(s)\frac{3\alpha_{11}-2\alpha_{13}-4\alpha_{22}+3\alpha_{33}}{8}\\&\quad+\left. h(s)^2\frac{\alpha_{11}-2\alpha_{12}+2\alpha_{23}-\alpha_{33}}{4}+h(s)^3\frac{\alpha_{11}\!-\!4\alpha_{12}+\!2\alpha_{13}-\!4\alpha_{23}+\!4\alpha_{22}\alpha_{33}}{16}\right]\\&-\sqrt{2}s_2\sin\left(\sqrt{2}\arctan\left(\frac{s_2}{s_1}\right)\right)\left[\frac{\beta_1-\delta_1-(\beta_3-\delta_3)}{8}\right.\\&\quad\quad\left.+h(s)\frac{\beta_1-\delta_1-2(\beta_2-\delta_2)+\beta_3-\delta_3}{8}\right]\\&+\frac{(s_1)^2+(s_2)^2}{16}\gamma\sqrt{2}s_2\sin\left(\sqrt{2}\arctan\left(\frac{s_2}{s_1}\right)\right)\left[\frac{\alpha_{11}+2\alpha_{12}-2\alpha_{23}-\alpha_{33}}{4}\right.\\&\quad+h(s)\frac{3\alpha_{11}-2\alpha_{13}-4\alpha_{22}+3\alpha_{33}}{4}\\&\quad+h(s)^2\frac{3(\alpha_{11}\!-2\alpha_{12}+\alpha_{23}\!-\alpha_{33})}{4}\\&\left.\quad+h(s)^3\frac{\alpha_{11}-4\alpha_{12}+2\alpha_{13}-4\alpha_{23}+4\alpha_{22}+\alpha_{33}}{4}\right].\ean We have similar formulas for $q_2$ and when $s_1\leq0$.

\subsection{Proof of Proposition \ref{PropQnorme}}\label{appprop}

Now $F(s)=\vert\nabla Q(s)\vert^2-\Delta Q(s)=(q_1(s))^2+(q_2(s))^2-\frac{\partial q_1}{\partial s_1}(s)-\frac{\partial q_2}{\partial s_2}(s).$
Besides, note that under \eqref{hyp1}, $\frac{\alpha_{11}+4\alpha_{12}+2\alpha_{13}+4\alpha_{23}+4\alpha_{22}+\alpha_{33}}{16}>0$. Therefore using Equations \eqref{formuleq1} and \eqref{formuleq2} we easily obtain that there exists a positive constant $C_1$ such that $(q_1(s))^2+(q_2(s))^2\geq C_1((s_1)^2+(s_2)^2)^3$. Finally, from Equation \eqref{formuleq1app}, we obtain after some calculations that there exists a positive constant $C_2$ such that $\frac{\partial q_1}{\partial s_1}(s)+\frac{\partial q_2}{\partial s_2}(s)\leq C_2((s_1)^2+(s_2)^2)^2$. Therefore Proposition \ref{PropQnorme} is true if $s_1\geq0$. If $s_1\leq0$, the result is true as well by symmetry.

\bigskip
\textbf{Acknowledgements:} I fully thank my Phd director Sylvie M\'el\'eard for suggesting me this research subject, and for her continual guidance during my work. I would also like to thank Denis Villemonais for his code and help for the simulation results. This article benefited from the support of the ANR MANEGE (ANR-09-BLAN-0215) and from the Chair ``Mod\'elisation
Math\'ematique et Biodiversit\'e" of Veolia Environnement - \'Ecole
Polytechnique - Museum National d'Histoire Naturelle - Fondation X.

\bibliographystyle{plainnat}
\bibliography{C:/Users/Camille/Dropbox/Travail/Biblio/mabiblio}

\end{document}